\documentclass[12pt,oneside]{amsart}
\usepackage{amssymb,amsfonts}
\usepackage{amsfonts,amsmath,amsxtra,amsthm,amssymb,latexsym,mathrsfs,srcltx,dsfont,eufrak,euscript}
\usepackage{amscd}
\usepackage{graphicx}
\usepackage[cp1251]{inputenc}
\usepackage[english]{babel}

\usepackage{pstricks}
\usepackage{pst-poly}
\usepackage{pst-node}
\usepackage{pst-blur}
\usepackage{enumerate}

\usepackage{hhline}

\usepackage{indentfirst}

\graphicspath{{noiseimages/}}

\DeclareMathOperator{\diag}{diag}

\newtheorem*{defuk*}{Означення}
\newtheorem*{thmuk*}{Теорема}
\newtheorem*{lemuk*}{Лема}
\newtheorem*{statuk*}{Твердження}
\newtheorem*{propuk*}{Пропозиція}
\newtheorem*{propertyuk*}{Властивість}
\newtheorem*{remuk*}{Зауваження}
\newtheorem*{exampleuk*}{Приклад}
\newtheorem*{consequk*}{Наслідок}
\newtheorem*{alguk*}{Алгоритм}

\newtheorem{definition}{Definition}
\newtheorem{theorem}{Theorem}
\newtheorem{corollary}{Corollary}
\newtheorem{lemma}{Lemma}
\newtheorem{remark}{Remark}

\usepackage[left=2.5cm,right=2cm,
    top=2cm,bottom=2cm,bindingoffset=0cm]{geometry}

\makeatletter
\def\@makefnmark{\hbox{\@textsuperscript{\normalfont\@thefnmark)}}}
\makeatother

\def\R{\mathbb R}
\def\C{\mathbb C}
\def\N{\mathbb N}

\def\L{\mathbb L}

\def\X{\mathcal X}

\def\ltri|{\left|\!\!\!\!\!\!\;\;\;\left|\!\!\!\!\!\!\;\;\;\left|}
\def\rtri|{\right|\!\!\!\!\!\!\;\;\;\right|\!\!\!\!\!\!\;\;\;\right|}

\numberwithin{equation}{section}
\numberwithin{thmuk}{section}
\numberwithin{remuk}{section}
\numberwithin{lemuk}{section}
\numberwithin{defuk}{section}
\numberwithin{statuk}{section}
\numberwithin{consequk}{section}
\numberwithin{alguk}{section}

\begin{document}

\author{Volodymyr L. Makarov}
\author{Denys V. Dragunov}
\address[V.L. Makarov and D.V. Dragunov]{Department of Numerical Mathematics\\ Institute of  Mathematics of NAS of Ukraine \\
3 Tereshchenkivs'ka Str., Kyiv-4, 01601, Ukraine}
   \email[V.L. Makarov]{makarov@imath.kiev.ua}
   \email[D.V. Dragunov]{dragunovdenis@gmail.com}

\title[Stability preserving structural transformations]{Stability preserving  structural transformations of systems of linear second-order ordinary differential equations}

\date{\today}

\dedicatory{Dedicated to the blessed memory of Professor Vladimir N. Koshlyakov}

\begin{abstract}
  In the paper we have developed a theory of stability preserving structural transformations of systems of second-order ordinary differential equations (ODEs), i.e., the transformations which preserve the property of Lyapunov stability. The main Theorem proved in the paper can be viewed  as an analogous of the Erugin's theorem for the systems of second-order ODEs. The Theorem allowed us to generalize the 3-rd and 4-th Kelvin -- Tait -- Chetayev theorems. The obtained theoretical results were successfully applied to the stability investigation of the rotary motion of a rigid body suspended on a string.
\end{abstract}

\maketitle
\section{Introduction}
In the present paper we concentrate our attention on the following system of second-order ordinary differential equations (ODEs):
\begin{equation}\label{Ko_Ma_De_1.1n}
  J\left(t\right)\ddot{x}\left(t\right)+\left(D\left(t\right)+G\left(t\right)\right)\dot{x}\left(t\right)+
  \left(P\left(t\right)+\Pi\left(t\right)\right)x\left(t\right)=F\left(t,x\left(t\right)\right)
\end{equation}
where
$x\left(t\right)=\verb"col"\left[x_{1}\left(t\right),x_{2}\left(t\right),
\ldots,x_{m}\left(t\right)\right]$
is an unknown vector-function. It is well known that a great number of dynamical systems can be approximately described by system \eqref{Ko_Ma_De_1.1n}. From the physical point of view the matrix $J\left(t\right)=J^{T}\left(t\right)>0$ (the upper index $T$ denotes the operation of transposition) describes the inertia characteristics of a dynamical system; the matrices $D\left(t\right)=D^{T}\left(t\right),
G\left(t\right)=-G^{T}\left(t\right),
P\left(t\right)=-P^{T}\left(t\right)$ and $\Pi\left(t\right)=\Pi^{T}\left(t\right)$ represent
a dissipative, gyroscopic, non-conservative positional and potential forces respectively; the vector-function $F\left(t,x\left(t\right)\right)$ represents an external forces acting on the system. We assume that all matrices are of dimension $m\times m$ and their elements are real-valued functions of $t \in \left[0, \infty\right).$
Also, we assume that the vector-function $F\left(t,x\left(t\right)\right)$ satisfies the condition $\|F\left(t,x\right)\|=O(\|x\|^2).$ System \eqref{Ko_Ma_De_1.1n} also will be referenced to as the second-order matrix differential equation.

As it was pointed out in \cite{Koshlyakov_1997}, when the gyroscopic terms in system
\eqref{Ko_Ma_De_1.1n} are periodic in $t$ with some period
$\tau>0$ then the formal application of the averaging method to the system could result in the discarding of gyroscopic structures though this structures had some stabilizing effect on the system before averaging.
Thus, it is very desirable to have a theoretical framework which allows us to transform the initial system into the system possessing the same stability properties and containing no gyroscopic structures.

A similar problem can be stated regarding the non-conservative positional structures $P\left(t\right).$ An elimination of the non-conservative positional structures from system \eqref{Ko_Ma_De_1.1n} without changing its stability properties has a particular interest in the case when system \eqref{Ko_Ma_De_1.1n} is autonomous, that is,
\begin{equation}\label{Ko_Ma_De_1.1}
    J\ddot{x}(t)+\left(D+G\right)\dot{x}(t)+\left(P+\Pi\right)x(t)=0\footnote{For simplisity the vector function of external forces was not taken into account.},
\end{equation}
where
$x=\verb"col"\left[x_{1}\left(t\right),x_{2}\left(t\right),
\ldots,x_{m}\left(t\right)\right]$ is an unknown vector; here again the matrix $J=J^{T}>0$  describes the inertia characteristics of the dynamical system and matrices $D=D^{T},\; G=-G^{T},\;
\Pi=\Pi^{T},\; P=-P^{T} $ represent
a dissipative, gyroscopic, non-conservative positional and potential forces respectively. The all matrix coefficients in equation \eqref{Ko_Ma_De_1.1} are assumed to be constant real matrices of dimension $m\times m.$

Suppose that the matrix $L(t)$ is {\it a Lyapunov matrix} (see \cite[p. 117]{Gantmaher_v2}). In several cases the application of a substitution
\begin{equation}\label{Ko_Ma_De_1.2}
  x(t)=J^{-\frac{1}{2}}L(t)\xi(t)
\end{equation}
  to equation \eqref{Ko_Ma_De_1.1} could lead us to the autonomous equation\footnote{In what follows,
saying that equation \eqref{Ko_Ma_De_1.3} is autonomous we also will bear in mind that the elements of the matrices $V_{1},\;W_{1}$
are real.}
\begin{equation}\label{Ko_Ma_De_1.3}
    \ddot{\xi}\left(t\right)+V_{1}\dot{\xi}\left(t\right)+W_{1}\xi\left(t\right)=0
\end{equation}
which do not contain a gyroscopic ($V=V^{T}$) and/or nonconservative positional structures ($W=W^{T}$). Since the matrix $L(t)$ is a Lyapunov matrix, the null solutions of systems \eqref{Ko_Ma_De_1.1} and \eqref{Ko_Ma_De_1.3} are stable, asymptotically stable or unstable simultaneously. On the other hand, because of the symmetrical properties of system \eqref{Ko_Ma_De_1.3} the stability investigation of its null solution can be an easier task than the stability investigation of the null solution of system \eqref{Ko_Ma_De_1.1}.

 The approach to the stability investigation of the second-order system of ODEs \eqref{Ko_Ma_De_1.1} which consists of reducing the initial problem to the problem of stability investigation of the corresponding equivalent (in the sense of Lyapunov, see \cite[p. 118]{Gantmaher_v2}) symmetric system \eqref{Ko_Ma_De_1.3}, was first suggested by D. L. Mingori  in \cite{Mingori_Eng}. He has shown that such approach can be very useful and fruitful for the stability investigation in analytical mechanics. In \cite{Mingori_Eng} the author has considered the case when $D> 0$ only. Though in \cite{Muller} the results of D. L. Mingori were extended on the case when $D\geq 0$ the necessary and sufficient conditions providing that a given non-symmetric second-order system of ODEs is equivalent in the sense of Lyapunov to some symmetric second-order system remains unknown: both papers \cite{Mingori_Eng} and \cite{Muller} contain sufficient conditions only.

Later the necessary and sufficient conditions providing that the autonomous system \eqref{Ko_Ma_De_1.1} can be reduced to some other autonomous system \eqref{Ko_Ma_De_1.3} with $W_{1}=W_{1}^{T}$ via substitution
\eqref{Ko_Ma_De_1.2} were found in papers \cite{Kosh_Makarov_UMJ_2000} and \cite{Kosh_Makarov_PMM_2001}. However, the results of that papers were obtained under additional assumptions that
$$G=H\hat{G},$$
\begin{equation}\label{Ko_Ma_De_1.3.1}
   \frac{d L\left(t\right)}{dH}=0, \forall t \geq 0,
\end{equation}
$$D>0, \verb"det"\left(G\right)\neq 0,$$
where $H$ denotes a positive numerical parameter.

In some cases the parameter $H$ can be a part of matrix $\Pi.$ This will be the case when equation
\eqref{Ko_Ma_De_1.1} describes a perturbed motion of a gyroscopic systems installed on the platform which rotates around the vertical with the angular velocity $\omega.$ Using assumptions \eqref{Ko_Ma_De_1.3.1} and assuming that $\Pi=\Pi^{(0)}+H\Pi^{(H)},$ where matrices
$\Pi^{(0)}, \Pi^{(H)}$ are independent on $H,$
the necessary and sufficient conditions providing the reducibility  of system \eqref{Ko_Ma_De_1.1} to some other system \eqref{Ko_Ma_De_1.3} with
$W_{1}=W_{1}^{T}$ where obtained in \cite{Kosh_Makarov_PMM_2007}.

In the present paper without any additional assumptions we have obtained the necessary and sufficient conditions (in terms of the matrix coefficients) providing that a given system of second-order ODEs is equivalent in the sense of Lyapunov to some other system of second-order ODEs with symmetric matrix coefficients. We have considered both the autonomous and non-autonomous cases. In the case when the initial system is autonomous we require that the reduced system be autonomous too.

The paper is organized as follows.

In Section \ref{NON_A_section} we introduce the notion of the {\it structural transformation} of a system of second-order ODEs and give the definition of the $L_{k}$-equivalent systems of second-order ODEs. Using the notion of the $L_{k}$-equivalence we formulate two symmetrization problems for the non-autonomous system of second-order ODEs: the {\it problem of Elimination of Gyroscopic Structures} (EGS problem) and the {\it problem of Elimination of Non-conservative Positional Structures} (ENPS problem). In the section the necessary and sufficient conditions providing the solvability of the both problems where obtained.

In Section \ref{A_Section} we reformulate the EGS and ENPS problems for the case of the autonomous systems and introduce the notion of the $L$-equivalence of two autonomous systems of second-order ODEs. Theorem \ref{T_main_theorem} proved in the section can be considered as an analogous of the Erugin's theorem (see \cite[p. 121]{Gantmaher_v2}) for the autonomous systems of second-order
ODEs. Some useful  consequences from Theorem \ref{T_main_theorem} are stated in Section \ref{Connection_Section}. Among them there is a theorem which generalizes  the theorems of Mingori (see \cite{Mingori_Eng}) and M\"{u}ller (see \cite{Muller}).

In Section \ref{Connection_Section} we discuss the question of the interconnection between the notions of the $L_{k}$-equivalence and equivalence in the sense of Lyapunov.

In Section \ref{Priklad_Section} we demonstrate how the using of structural transformations can facilitate the stability investigation of the null solution of the autonomous second-order system of ODEs describing the rotary motion of a rigid body suspended on a string.

Section \ref{Concl_Section} contains several conclusions about the theoretical results presented in the paper.

\section{Structural transformations of the non-autonomous systems of second-order ordinary differential equations}\label{NON_A_section}

Let us consider the following system of second-order ordinary differential equations:
\begin{equation}\label{Chapt_2_eq_0}
    \mathbf{\ddot{x}}+A\left(t\right)\mathbf{\dot{x}}+B\left(t\right)\mathbf{x}=0,
\end{equation}
where $\mathbf{x}=\overrightarrow{x}\left(t\right)=\left[x_{1}\left(t\right),\ldots, x_{m}\left(t\right)\right]^{T}$ is an unknown vector-function. By default, we assume that $A\left(t\right), B\left(t\right)$ are square matrices of order $m$ whose elements are continuous on $\left[t_{0}, +\infty\right)$ functions, i.e.,  $A\left(t\right), B\left(t\right)\in M_{m}\left(C\left[t_{0}, +\infty\right)\right).$ Also we will use the notation $M_{m}\left(C^{i}\left[t_{0}, +\infty\right)\right),$ $i=1,2$ to denote the linear spaces of square matrices of order $m$ whose elements belong to the functional space $C^{i}\left[t_{0}, +\infty\right),$ $i=1,2,$ and the notation $M_{m,n}(\mathbb{R})$ will be used to denote the space of constant real matrices of dimension $m\times n.$
\begin{definition}
The structural transformation of the second-order system of ordinary differential equations \eqref{Chapt_2_eq_0} is the transformation of unknown vector $\mathbf{x}$ which can be expressed in the form
\begin{equation}\label{Chapt_2_eq_4}
    \mathbf{x}=L\left(t\right)\mathbf{\xi},
\end{equation}
 where $\mathbf{\xi}=\left[\xi_{1}\left(t\right), \ldots, \xi_{m}\left(t\right)\right]^{T}$ is a new unknown vector-function, $L\left(t\right)\in M_{m}\left(C^{2}\left[t_{0}, +\infty\right)\right),$ $\det\left(L\left(t\right)\right)\neq 0,\; \forall t\in\left[t_{0}, +\infty\right).$
\end{definition}

 Applying transformation \eqref{Chapt_2_eq_4} to system \eqref{Chapt_2_eq_0} we obtain the following system of second-order ordinary differential equations:
\begin{equation}\label{Chapt_2_eq_6}
\begin{array}{c}
  L\left(t\right)\mathbf{\ddot{\xi}}\left(t\right)+\left(2\dot{L}\left(t\right)+A\left(t\right)L\left(t\right)\right)\mathbf{\dot{\xi}}\left(t\right)+ \\[1.2em]
  +\left(\ddot{L}\left(t\right)+A\left(t\right)\dot{L}\left(t\right)+B\left(t\right)L\left(t\right)\right)\mathbf{\xi}\left(t\right)= 0, \end{array}
\end{equation}
or , in more convenient form,
\begin{equation}\label{Chapt_2_eq_1}
    \mathbf{\ddot{\xi}}+V\left(t\right)\mathbf{\dot{\xi}}+W\left(t\right)\xi=0,
\end{equation}
where
\begin{equation}\label{Chapt_2_eq_2}
   \begin{array}{c}
    V\left(t\right)=L^{-1}\left(t\right)\left(2\dot{L}\left(t\right)+A\left(t\right)L\left(t\right)\right),\\[1.2em]
    W\left(t\right)=L^{-1}\left(t\right)\left(\ddot{L}\left(t\right)+A\left(t\right)\dot{L}\left(t\right)+B\left(t\right)L\left(t\right)\right).\\
    \end{array}
\end{equation}
Apparently we have that $V\left(t\right), W\left(t\right)\in M_{m}\left(C\left[t_{0}, +\infty\right)\right).$ Therefore, applying transformation \eqref{Chapt_2_eq_4} to the system with continuous on $[t_{0}, +\infty)$ matrix coefficients \eqref{Chapt_2_eq_0},  we arrive at system \eqref{Chapt_2_eq_1} whose matrix coefficients are continuous on  $\left[t_{0}, +\infty\right)$ too.

\begin{definition}\label{O_pro_Lk_ekvival}
We say that the system of second-order ODEs \eqref{Chapt_2_eq_0} is $L_{k}$-equivalent to system \eqref{Chapt_2_eq_1} {\normalfont ($k\in\{0, 1, 2\}$)} if there exists a matrix $L\left(t\right)\in M_{m}\left(C^{2}\left[t_{0}, +\infty\right)\right)$ satisfying conditions
\begin{enumerate}
\item\label{Ozn_M_L_2p_nmova_1} $\left|\det\left(L\left(t\right)\right)\right|>\eta>0,$ $\forall t\in\left[t_{0}, +\infty\right],$
\item\label{Ozn_M_L_2p_nmova_2}  $\sup\limits_{t\in\left[t_{0}, +\infty\right)}\bigg\|\cfrac{d^{i}}{d t^{i}}L\left(t\right)\bigg\|<+\infty,$ $\forall i\in \overline{0, k},$
\end{enumerate}
together with equalities  \eqref{Chapt_2_eq_2}.
          A matrix $L\left(t\right)\in M_{m}\left(C^{2}\left[t_{0}, +\infty\right)\right)$ which satisfy conditions \ref{Ozn_M_L_2p_nmova_1}, \ref{Ozn_M_L_2p_nmova_2} for some $k\in\left\{0,1,2\right\}$ is called an $L_{k}$-matrix.
\end{definition}
 According to the definition given in \cite[p. 353]{Demidovich}, a matrix $L(t)\in M_{m}(C^{1}[t_{0}, +\infty))$ which satisfy conditions \ref{Ozn_M_L_2p_nmova_1}, \ref{Ozn_M_L_2p_nmova_2} for $k=0,$ is called a {\it regular on $\left[t_{0}, +\infty\right)$ matrix}. \label{Page_reg_matr} Transformation \eqref{Chapt_2_eq_4}, where $L\left(t\right)$ is an $L_{2}$-matrix can also be referenced to as a {\it Lyapunov transformation of system of second-order ODEs} (compare with the definition of a {\it Lyapunov transformation} form \cite[p. 116]{Gantmaher_v2}).

Let us consider the following symmetrization problems for the given system of second-order ODEs
 \eqref{Chapt_2_eq_0}:
\begin{enumerate}\label{Page_zadacha_1}
\item\label{Z_1} the {\it problem of Elimination of Gyroscopic Structures} (EGS problem)  which consists in finding an $L_{k}$-matrix $L\left(t\right)$ ($k=0,1,2$) together with matrices $V\left(t\right), W\left(t\right) \in M_{m}\left(\C\left[t_{0}, +\infty\right)\right),$ $V\left(t\right)=V^{T}\left(t\right),$ such that equalities \eqref{Chapt_2_eq_2} hold true $\forall t\in\left[t_{0}, +\infty\right);$
\item\label{Z_2} the {\it problem of Elimination of Non-conservative Positional Structures} (ENPS problem) which consists in finding an $L_{k}$-matrix $L\left(t\right)$ ($k=0,1,2$) together with matrices $V\left(t\right), W\left(t\right) \in M_{m}\left(\C\left[t_{0}, +\infty\right)\right),$ $W\left(t\right)=W^{T}\left(t\right),$ such that equalities \eqref{Chapt_2_eq_2} hold true $\forall t\in\left[t_{0}, +\infty\right).$
\end{enumerate}
If the matrices $L\left(t\right), V\left(t\right), W\left(t\right)$ mentioned in items \ref{Z_1} and/or \ref{Z_2} exist then we say that the EGS and/or ENPS problems for system \eqref{Chapt_2_eq_0} can be solved by means of $L_{k}$-transformation.

Both symmetrization problems can be stated in terms of the $L_{k}$-equivalence in the following way:
\begin{enumerate}
\item to find a system \eqref{Chapt_2_eq_1} which is $L_{k}$-equivalent to the given system \eqref{Chapt_2_eq_0} and such that  $V\left(t\right)=V^{T}\left(t\right)$ (EGS problem);
\item to find a system \eqref{Chapt_2_eq_1} which is $L_{k}$-equivalent to the given system \eqref{Chapt_2_eq_0} and such that $W\left(t\right)=W^{T}\left(t\right)$ (ENPS problem).
\end{enumerate}

Let us find the necessary and sufficient conditions
(in terms of matrices $A\left(t\right), B\left(t\right)$)
 providing the solvability of the EGS and/or ENPS problems for the given system \eqref{Chapt_2_eq_0}, or, in other words, the necessary and sufficient conditions providing that system \eqref{Chapt_2_eq_0} is $L_{k}$-equivalent to some system \eqref{Chapt_2_eq_1} with $V\left(t\right)=V^{T}\left(t\right)$ and/or $W\left(t\right)=W^{T}\left(t\right)$ for some $k=0,1,2$.

 Supposing that the matrix coefficient in front of the vector-function $\mathbf{\dot{\xi}}$ in system \eqref{Chapt_2_eq_1} is symmetric (i.e., there is no gyroscopic structures), we arrive at the following matrix differential equation with respect to the unknown $L_{k}$-matrix $L\left(t\right)$:
\begin{equation}\label{Chapt_2_eq_7}
\begin{array}{c}
  2\left(\dot{L}\left(t\right)L^{T}\left(t\right)-L\left(t\right)\dot{L}^{T}\left(t\right)\right)+A\left(t\right)L\left(t\right)L^{T}\left(t\right)- \\[1.2em]
  -L\left(t\right)L^{T}\left(t\right)A^{T}\left(t\right)=0. \\
\end{array}%
\end{equation}
Similarly to that, assuming that the matrix coefficient in front of the vector-function $\mathbf{\xi}$ in system \eqref{Chapt_2_eq_1} is symmetric (i.e., there is no non-conservative positional structures) we arrive at the equation
\begin{equation}\label{Chapt_2_eq_8}
\begin{array}{c}
  \ddot{L}\left(t\right)L^{T}\left(t\right)-L\left(t\right)\ddot{L}^{T}\left(t\right)+A\left(t\right)\dot{L}\left(t\right)L^{T}\left(t\right)- \\[1.2em]
  -L\left(t\right)\dot{L}^{T}\left(t\right)A^{T}\left(t\right)+
  B\left(t\right)L\left(t\right)L^{T}\left(t\right)-L\left(t\right)L^{T}\left(t\right)B^{T}\left(t\right)=0.
  \end{array}
\end{equation}

 It is easy to verify that there exists a unique pair of matrices $K(t),$  $S(t),$ such that
\begin{equation}\label{Chapt_2_eq_9}
    \dot{L}L^{T}=\dot{L}\left(t\right)L^{T}\left(t\right)=K\left(t\right)+S\left(t\right),\quad K\left(t\right)=-K^{T}\left(t\right),\quad  S\left(t\right)=S^{T}\left(t\right).
\end{equation}
If matrix $L\left(t\right)$ is an $L_{k}$-matrix ($k=0,1,2$) then matrices $K\left(t\right)$ and $S\left(t\right)$ \eqref{Chapt_2_eq_9} belongs to $M_{m}(C^{1}[t_{0}, +\infty)).$ It is easy to see that
\begin{equation}\label{Chapt_2_eq_10}
    \frac{d}{d t}\left(L\left(t\right)L^{T}\left(t\right)\right)=\dot{L}\left(t\right)L^{T}\left(t\right)+L\left(t\right)\dot{L}^{T}\left(t\right)=2S\left(t\right),
\end{equation}
and
\begin{equation}\label{Chapt_2_eq_11}
    L\left(t\right)L^{T}\left(t\right)=2\int\limits_{t_{0}}^{t}S\left(\nu\right)d\nu+S_{0}, \quad L\left(t_{0}\right)L^{T}\left(t_{0}\right)=S_{0}=S_{0}^{T}>0.
\end{equation}

Taking into account equalities \eqref{Chapt_2_eq_9}, \eqref{Chapt_2_eq_11}, we can rewrite equations \eqref{Chapt_2_eq_7} and (\ref{Chapt_2_eq_8}) in the form of
\begin{equation}\label{Chapt_2_eq_12}
\begin{array}{c}
  4K\left(t\right)+A\left(t\right)\bigg(2\int\limits_{t_{0}}^{t}S\left(\nu\right)d\nu+S_{0}\bigg)-\bigg(2\int\limits_{t_{0}}^{t}S\left(\nu\right)d\nu+S_{0}\bigg)A^{T}\left(t\right)=0
\end{array}
\end{equation}
and
\begin{equation}\label{Chapt_2_eq_13}
\begin{array}{c}
  2\dot{K}\left(t\right)+A\left(t\right)\left(S\left(t\right)+K\left(t\right)\right)
  -\left(S\left(t\right)-K\left(t\right)\right)A^{T}\left(t\right)+\\[1.2em]
  +B\left(t\right)\bigg(\int\limits_{t_{0}}^{t}2S\left(\nu\right)d\nu+S_{0}\bigg)
  -\bigg(\int\limits_{t_{0}}^{t}2S\left(\nu\right)d\nu+S_{0}\bigg)B^{T}\left(t\right)=0\\[1.2em]
\end{array}%
\end{equation}
respectively. The question arises: what necessary and sufficient requirements have to be imposed on the matrices $K\left(t\right)$ and $S\left(t\right)$ to provide the existence of an $L_{k}$-matrix $L\left(t\right)$ which satisfy equality \eqref{Chapt_2_eq_9}? The answer to this question is given by the following theorem.
\begin{theorem}\label{Teor_pro_dobutok}
  A regular on $\left[t_{0}, +\infty\right)$ matrix $L\left(t\right)$ which satisfies equality \eqref{Chapt_2_eq_9} exists if and only if the matrices $K\left(t\right), S\left(t\right)$ belong to $M_{m}\left(C\left[t_{0}, +\infty\right)\right)$ and satisfy the following inequalities:
  \begin{equation}\label{Chapt_2_eq_14}
     \bigg|2\int\limits_{t_{0}}^{t}Tr\left(S\left(\nu\right)\right)d\nu+Tr\left(S_{0}\right)\bigg|\leq \mu^{2}, \forall t\in\left[t_{0}, +\infty\right),
  \end{equation}
  \begin{equation}\label{Chapt_2_eq_15}
     \det\bigg(2\int\limits_{t_{0}}^{t}S\left(\nu\right)d\nu+S_{0}\bigg)\geq \eta^{2}, \forall t\in\left[t_{0}, +\infty\right)
  \end{equation}
  for some constants $\mu>0,\;\eta>0$ and real valued positive definite symmetric matrix $S_{0}\in M_{m}\left(\R\right).$
\end{theorem}
  \begin{proof} {\it Necessity.} Suppose that there exists a matrix $L\left(t\right)$ which belongs to $M_{m}\left(C^{1}\left[t_{0}, +\infty\right)\right)$ and satisfies equality \eqref{Chapt_2_eq_9} together with inequalities
  \begin{equation}\label{Chapt_2_eq_16}
    \left\|L\left(t\right)\right\|_{F}\leq \mu,\quad \forall t\in \left[t_{0}, +\infty\right)\footnote{Here $\left\|A\right\|_{F}$ denotes the Frobenius norm of matrix $A,$ that is, $\left\|A\right\|_{F}=\sqrt{Tr(AA^{T})}.$},
  \end{equation}
  \begin{equation}\label{Chapt_2_eq_17}
    \left|\det\left(L\left(t\right)\right)\right|\geq \eta,\quad \forall t\in \left[t_{0}, +\infty\right),
  \end{equation}
  for some constants $\mu>0,$ $\eta>0.$
 It easy to see that the matrices $K\left(t\right), S\left(t\right)$ appearing in \eqref{Chapt_2_eq_9} belong to $M_{m}\left(C\left[t_{0}, +\infty\right)\right),$ and the necessity of conditions \eqref{Chapt_2_eq_14}, \eqref{Chapt_2_eq_15} immediately follows from  \eqref{Chapt_2_eq_11}. The necessity in the theorem is proved.

  {\it Sufficiency .} Suppose that $K\left(t\right), S\left(t\right)\in M_{m}(C[t_{0}, +\infty)),$ $K(t)=-K^{T}(t),$ $S(t)=S^{T}(t)$ and inequalities \eqref{Chapt_2_eq_14}, \eqref{Chapt_2_eq_15} hold true for some constants  $\mu>0,\;\eta>0$ and some positive definite symmetric matrix  $S_{0}.$ Assuming that the matrix $L=L\left(t\right)$ satisfies equality \eqref{Chapt_2_eq_9} $\forall t\in\left[t_{0}, +\infty\right)$ together with the initial condition
  \begin{equation}\label{Chapt_2_eq_18}
    L\left(t_{0}\right)=L_{0},\quad L_{0}L_{0}^{T}=S_{0},
  \end{equation}
  we arrive at the conclusion that equality \eqref{Chapt_2_eq_11} together with inequality \eqref{Chapt_2_eq_14} imply inequality \eqref{Chapt_2_eq_16} as well as inequality \eqref{Chapt_2_eq_15} implies inequality \eqref{Chapt_2_eq_17}.

  Let us prove that the solution $L=L\left(t\right)$ to the Cauchy problem \eqref{Chapt_2_eq_9}, \eqref{Chapt_2_eq_18} supplemented by conditions \eqref{Chapt_2_eq_14}, \eqref{Chapt_2_eq_15} exists and is unique on $\left[t_{0}, T\right]$ for any arbitrary  $T>t_{0}.$
If we denote by $\lambda_{i},$ $i=1,2,\ldots, m$ the ascending ordered eigenvalues of matrix $S_{0}$, that is,  $0<\lambda_{1}\leq\lambda_{2}\leq\ldots\leq\lambda_{m},$ then inequality \eqref{Chapt_2_eq_14} implies that
\begin{equation}\label{Chapt_2_eq_19}
    \lambda_{m}\leq \sum\limits_{i=1}^{m}\lambda_{i}=Tr\left(S_{0}\right)\leq \mu^{2}.
\end{equation}
Taking into account inequality \eqref{Chapt_2_eq_19} we can obtain from inequality \eqref{Chapt_2_eq_15} the estimate
$$\lambda_{1}=\cfrac{\det\left(S_{0}\right)}{\lambda_{2}\ldots\lambda_{m}}\geq \cfrac{\eta^{2}}{\mu^{2\left(m-1\right)}}$$
which leads us to the inequality
$$\left\|L^{-1}_{0}\right\|^{2}_{E}=Tr\left(S^{-1}_{0}\right)\leq \cfrac{m}{\lambda_{1}}\leq\cfrac{m\mu^{2\left(m-1\right)}}{\eta^{2}}=\kappa^{2}.$$
Since $\det(L(t))\neq 0,$ $\forall t\in [t_{0}, +\infty),$ equality \eqref{Chapt_2_eq_9} can be rewritten in the form of
\begin{equation}\label{Chapt_2_eq_20}
    \dot{L}=F\left(t, L\right)=\left(S\left(t\right)+K\left(t\right)\right)\left(L^{-1}\right)^{T}.
\end{equation}
Now we intend to show that the matrix-valued function $F\left(t, L\right)$ satisfies conditions of the Picard–Lindel\"{o}f theorem (see, for example, \cite[p. 8]{Hartman_DU}) in the rectangle
\begin{equation}\label{Chapt_2_eq_21}
    \mathfrak{P}=\bigg\{\left(t, L\right)\in \R\times M_{m}\left(\R\right): t_{0}\leq t\leq T,\quad \left\|L-L_{0}\bigg\|_{E}\leq\frac{\delta}{\kappa},\;0<\delta<1\right\}.
\end{equation}
  Taking into account that the elements of matrix-functions $S\left(t\right)$ and $K\left(t\right)$ are continuous on $[t_{0}, +\infty),$ it remains only to show that the matrix-valued function $F\left(t, L\right)$ is Lipschitz-continuous on $\mathfrak{P}$ \eqref{Chapt_2_eq_21} with respect to its second argument $L.$ This fact follows from the following inequalities, which are valid for any matrices  $L_{i}\in M_{m}\left(\R\right), i=1,2,$ such that $\left\|L_{i}-L_{0}\right\|_{E}\leq \cfrac{\delta}{\kappa}:$
$$\left\|L^{-1}_{1}-L^{-1}_{2}\right\|_{E}=\left\|\left(L_{1}-L_{0}+L_{0}\right)^{-1}-\left(L_{2}-L_{0}+L_{0}\right)^{-1}\right\|_{E}=$$
$$=\left\|L_{0}^{-1}\left(\left(L_{1}-L_{0}\right)L^{-1}_{0}+E\right)^{-1}-L_{0}^{-1}\left(\left(L_{2}-L_{0}\right)L^{-1}_{0}+E\right)^{-1}\right\|_{E}=$$
$$=\left\|L_{0}^{-1}\left(\sum_{i=0}^{\infty}\left(-1\right)^{i}\Bigl(\left(L_{1}-L_{0}\right)L^{-1}_{0}\Bigr)^{i}-\sum_{i=0}^{\infty}\left(-1\right)^{i}\Bigl(\left(L_{2}-L_{0}\right)L^{-1}_{0}\Bigr)^{i}\right)\right\|_{E}=$$
$$=\left\|L_{0}^{-1}\left(\sum_{i=1}^{\infty}\left(-1\right)^{i}\left(\Bigl(\left(L_{1}-L_{0}\right)L^{-1}_{0}\Bigr)^{i}-\Bigl(\left(L_{2}-L_{0}\right)L^{-1}_{0}\Bigr)^{i}\right)\right)\right\|_{E}\leq$$
$$\leq\left\|L_{0}^{-1}\right\|_{E}\sum_{i=1}^{\infty}\left(\sum_{j=1}^{i}\left\|L_{1}-L_{0}\right\|_{E}^{i-j}\left\|L_{1}-L_{2}\right\|_{E}\left\|L_{2}-L_{0}\right\|_{E}^{j-1}\left\|L^{-1}_{0}\right\|_{E}^{i}\right)=$$
$$=\left\|L_{0}^{-1}\right\|_{E}^{2}\left\|L_{1}-L_{2}\right\|_{E}\sum_{i=1}^{\infty}\left(\sum_{j=1}^{i}\left\|L_{1}-L_{0}\right\|_{E}^{i-j}\left\|L_{0}-L_{2}\right\|_{E}^{j-1}\left\|L^{-1}_{0}\right\|_{E}^{i-1}\right)\leq$$
\begin{equation}\label{Chapt_2_eq_22}
    \leq\left\|L_{1}-L_{2}\right\|_{E}\kappa^{2}\sum_{i=1}^{\infty}i\delta^{i-1}=\frac{\kappa^{2}}{\left(1-\delta\right)^{2}}\left\|L_{1}-L_{2}\right\|_{E}.
\end{equation}
In the above formula we have used the equality (see, for example, \cite[p. 113]{Gantmaher_v1}) $$\left(A+E\right)^{-1}=\sum\limits_{i=0}^{\infty}\left(-1\right)^{i}A^{i}, \quad \forall A\in M_{m}\left(\R\right), \; \left\|A\right\|<1,$$  and the evident identity $$A^{n}-B^{n}=\sum\limits_{i=1}^{n}A^{n-i}\left(A-B\right)B^{i-1},\quad \forall A,B \in M_{m}\left(\R\right),\;n=1,2,\ldots.$$
Using \eqref{Chapt_2_eq_22} we can estimate the norm of $F\left(t, L\right)$ on the rectangle $\mathfrak{P}$ \eqref{Chapt_2_eq_21} in the following way:
$$\max\limits_{\left\|L-L_{0}\right\|_{E}\leq\frac{\delta}{\kappa}}\left\|\left(L^{-1}\right)^{T}\right\|_{E}=\max\limits_{\left\|L-L_{0}\right\|_{E}\leq\frac{\delta}{\kappa}}\left\|L^{-1}-L_{0}^{-1}+L_{0}^{-1}\right\|_{E}\leq$$
$$\leq \frac{\kappa^{2}}{\left(1-\delta\right)^{2}}\max\limits_{\left\|L-L_{0}\right\|_{E}\leq\frac{\delta}{\kappa}}\left\|L-L_{0}\right\|_{E}+\left\|L_{0}^{-1}\right\|_{E}\leq \frac{\kappa\delta}{\left(1-\delta\right)^{2}}+\kappa,$$
$$\max\limits_{\left(t, L\right)\in \Pi}\left\|F\left(t, L\right)\right\|_{E}\leq \max\limits_{t\in\left[t_{0}, T\right]}\left\|K\left(t\right)+S\left(t\right)\right\|_{E}\max\limits_{\left\|L-L_{0}\right\|_{E}\leq\frac{\delta}{\kappa}}\left\|\left(L^{-1}\right)^{T}\right\|_{E}=F_{\mathfrak{P}}.$$

Thus, the conditions of the Picard–Lindel\"{o}f theorem are satisfied and the solution of the Cauchy problem \eqref{Chapt_2_eq_18}, \eqref{Chapt_2_eq_20} exists at least on the interval $I_{h}=\left[t_{0}, h\right],$ where $h=\min\bigg\{T, \cfrac{\delta}{\kappa F_{\mathfrak{P}}}\bigg\}.$ If $h=T$ then the theorem is proved. Otherwise, if $h<T$ then, applying the same reasoning as above to equation \eqref{Chapt_2_eq_20} with the initial condition $L_{h}=L\left(h\right),$ we arrive at the conclusion that the solution to the Cauchy problem \eqref{Chapt_2_eq_18}, \eqref{Chapt_2_eq_20} exists at least on the interval $\left[t_{0}, 2h\right].$ Apparently, after a finite number of iterations we will prove that the solution exists on $\left[t_{0}, T\right].$ From the arbitrariness of $T$ it follows that the solution to the Cauchy problem \eqref{Chapt_2_eq_18}, \eqref{Chapt_2_eq_20} exists on $\left[t_{0}, +\infty\right).$  The theorem is proved.
\end{proof}

  It is not hard to verify that the matrix   $K\left(t\right)+S\left(t\right)$ where $K\left(t\right)=-K^{T}\left(t\right), S\left(t\right)=S^{T}\left(t\right)$ is bounded on $[t_{0}, +\infty]$ and/or belongs to $M_{m}(C^{k}[t_{0}, +\infty))$ if and only if both of the two matrices $K(t)$ and $S(t)$ are bounded on $[t_{0}, +\infty)$ and/or belong to $M_{m}(C^{k}[t_{0}, +\infty)).$ Taking this fact into account and using Theorem \ref{Teor_pro_dobutok} we can make several conclusions stated below.
  \begin{corollary}\label{N_pro_dobutok}
An $L_{k}$-matrix $L\left(t\right)$ {\normalfont (}$k=1,2${\normalfont )} satisfying equality \eqref{Chapt_2_eq_9} exists if and only if $K\left(t\right), S\left(t\right)\in M_{m}\left(C^{1}\left[t_{0}, +\infty\right)\right)$ and the following conditions hold true:
\begin{enumerate}
\item there exist constants $\mu>0, \eta>0$ and matrix $S_{0}\in M_{m}\left(\R\right),$  $S_{0}=S_{0}^{T}>0$ satisfying inequalities  \eqref{Chapt_2_eq_14}, \eqref{Chapt_2_eq_15};
\item  $\sup\limits_{t\in\left[t_{0}, +\infty\right)}\bigg\|\cfrac{d^{i}}{d t^{i}}K\left(t\right)\bigg\|+\sup\limits_{t\in\left[t_{0}, +\infty\right)}\bigg\|\cfrac{d^{i}}{d t^{i}}S\left(t\right)\bigg\|<+\infty,$ $\forall i\in\overline{0, k-1}.$
\end{enumerate}
 \end{corollary}

Equation \eqref{Chapt_2_eq_12} and Corollary  \ref{N_pro_dobutok} imply the following theorem.
\begin{theorem}\label{T_Vikl_Giro_Struct}
The given system of second-order ODEs \eqref{Chapt_2_eq_0} with $A\left(t\right)\in M_{m}\left(C^{1}\left[t_{0}, +\infty\right)\right)$ is $L_{k}$-equivalent {\normalfont (}$k=0,1,2${\normalfont )} to some system \eqref{Chapt_2_eq_1} with $V\left(t\right)=V^{T}\left(t\right)$ if and only if there exist the symmetric matrices $S\left(t\right)\in M_{m}\left(C^{1}\left[t_{0}, +\infty\right)\right),$ $S_{0}\in M_{m}\left(\R\right),$ $S_{0}>0$ which define the skew-symmetric matrix $K\left(t\right)$
\begin{equation}\label{Chapt_2_eq_23}
    4K\left(t\right)=\Lambda\left(t\right)A^{T}\left(t\right)-A\left(t\right)\Lambda\left(t\right), \quad \Lambda\left(t\right)=2\int\limits_{t_{0}}^{t}S\left(\nu\right)d \nu+S_{0},
\end{equation}
  and satisfy conditions
\begin{enumerate}
\item\label{T_Vikl_Giro_Struct_U_1} \eqref{Chapt_2_eq_14}, \eqref{Chapt_2_eq_15} for some constants $\mu>0,\; \eta>0;$
\item\label{T_Vikl_Giro_Struct_U_2}  $\sup\limits_{t\in\left[t_{0}, +\infty\right)}\bigg\|\cfrac{d^{i}}{d t^{i}}K\left(t\right)\bigg\|+\sup\limits_{t\in\left[t_{0}, +\infty\right)}\bigg\|\cfrac{d^{i}}{d t^{i}}S\left(t\right)\bigg\|<+\infty,$ $\forall i\in\overline{0, k-1},\; (k\neq 0)\footnote{In the case when $k=0$ condition \ref{T_Vikl_Giro_Struct_U_2} should be neglected.}.$
\end{enumerate}
\end{theorem}
From Theorem \ref{T_Vikl_Giro_Struct} we obtain the following corollary.
\begin{corollary}\label{N_z_T_Vikl_Giro_Struct}
  The given system of second-order ODEs   \eqref{Chapt_2_eq_0} with $A\left(t\right)\in M_{m}\left(C^{1}\left[t_{0}, +\infty\right)\right)$ is always $L_{0}$-equivalent to some system \eqref{Chapt_2_eq_1} with $V\left(t\right)=V^{T}\left(t\right).$
\end{corollary}

From equation \eqref{Chapt_2_eq_13} and Corollary \ref{N_pro_dobutok} we can easily obtain the theorem which gives the necessary and sufficient conditions for solvability of the ENPS problem.
\begin{theorem}\label{T_Vikl_NP_Struct}
The given system of second-order ODEs  \eqref{Chapt_2_eq_0} is $L_{k}$-equivalent {\normalfont (}$k=0,1,2${\normalfont )} to some system \eqref{Chapt_2_eq_1} with $W\left(t\right)=W^{T}\left(t\right)$ if and only if there exist the symmetric matrices $S\left(t\right)\in M_{m}\left(C^{1}\left[t_{0}, +\infty\right)\right),$ $S_{0}\in M_{m}\left(\R\right),$ $S_{0}>0$
and the skew-symmetric matrix $K\left(t\right)$ which satisfy the matrix differential equation
\begin{equation}\label{Chapt_2_eq_24}
\begin{array}{c}
  2\dot{K}\left(t\right)+A\left(t\right)K\left(t\right)+K\left(t\right)A^{T}\left(t\right)+\\ [1.2em] +A\left(t\right)S\left(t\right) -S\left(t\right)A^{T}\left(t\right)
  +B\left(t\right)\Lambda\left(t\right)
  -\Lambda\left(t\right)B^{T}\left(t\right)=0,\\[1.2em]
  \Lambda\left(t\right)=2\int\limits_{t_{0}}^{t}S\left(\nu\right)d \nu+S_{0},
\end{array}%
\end{equation}
and conditions \ref{T_Vikl_Giro_Struct_U_1}, \ref{T_Vikl_Giro_Struct_U_2} of Theorem {\normalfont \ref{T_Vikl_Giro_Struct}}.
\end{theorem}

It is worth to emphasize that for any initial condition
$K\left(t_{0}\right)=K_{0}=-K^{T}_{0}\in M_{m}\left(\R\right)$ the solution $K\left(t\right)$ to the matrix differential equation \eqref{Chapt_2_eq_24} is a skew-symmetric matrix. Indeed, if we sum up equation \eqref{Chapt_2_eq_24} with the transposed equation \eqref{Chapt_2_eq_24} we obtain the Cauchy problem
\begin{equation}\label{Chapt_2_eq_24'}
\begin{array}{c}
  2\dot{N}\left(t\right)+A\left(t\right)N\left(t\right)+N\left(t\right)A^{T}\left(t\right)=0, \\[1.2em]
  N\left(t\right)=K\left(t\right)+K^{T}\left(t\right),\; N\left(0\right)=0.
\end{array}
\end{equation}
It is easy to see that the conditions of the Picard–Lindel\"{o}f theorem  for the Cauchy problem \eqref{Chapt_2_eq_24'} are fulfilled and its solution $N(t)$ exists and is unique on $\left[t_{0}, +\infty\right).$ Therefore, the problem has the trivial solution only, that is, $N\left(t\right)=0,\; \forall t\in\left[t_{0}, +\infty\right)$ and $K\left(t\right)=-K^{T}\left(t\right),\;\forall t\in\left[t_{0}, +\infty\right).$  Such conclusion can  also be obtained from the analysis of the analytical expression for the general solution  $K\left(t\right)$ of equation \eqref{Chapt_2_eq_24} (see, for example, \cite[p. 188]{mazko}).

From Theorem \ref{T_Vikl_NP_Struct} we can easily obtain the corollary.
\begin{corollary}\label{N_z_T_Vikl_NP_Struct}
  The given system of second-order ODEs \eqref{Chapt_2_eq_0} is always $L_{0}$-equivalent to some other system \eqref{Chapt_2_eq_1} with $W\left(t\right)=W^{T}\left(t\right).$
\end{corollary}

Combining Theorems \ref{T_Vikl_NPG_Struct} and \ref{T_Vikl_Giro_Struct} we arrive at the following one.
\begin{theorem}\label{T_Vikl_NPG_Struct}
The given system of second-order ODEs  \eqref{Chapt_2_eq_0} with $A\left(t\right)\in M_{m}\left( C^{1}\left[t_{0}, +\infty\right)\right)$ is $L_{k}$-equivalent {\normalfont (}$k=0,1,2${\normalfont )} to some system \eqref{Chapt_2_eq_1} with $V\left(t\right)=V^{T}\left(t\right),$ $W\left(t\right)=W^{T}\left(t\right)$ if and only if there exist the symmetric matrices  $S\left(t\right)\in M_{m}\left(C^{1}\left[t_{0}, +\infty\right)\right),$ $S_{0}\in M_{m}\left(\R\right),$ $S_{0}>0$ which define the skew-symmetric matrix $K\left(t\right)$ \eqref{Chapt_2_eq_23} and satisfy conditions \ref{T_Vikl_Giro_Struct_U_1}, \ref{T_Vikl_Giro_Struct_U_2} of Theorem {\normalfont \ref{T_Vikl_Giro_Struct}} together with equality
\begin{equation}\label{Chapt_2_eq_25}
\begin{array}{c}
    \Lambda\left(t\right)M^{T}\left(t\right)=M\left(t\right)\Lambda\left(t\right),\quad \forall t\in\left[t_{0}, +\infty\right), \\[1.2em]
    M\left(t\right)=\cfrac{1}{2}\cfrac{d}{d t}A\left(t\right)+\cfrac{1}{4}A^{2}\left(t\right)-B\left(t\right),\quad \Lambda\left(t\right)=2\int\limits_{t_{0}}^{t}S\left(\nu\right)d \nu+S_{0}.
  \end{array}
\end{equation}
\end{theorem}

 Condition \eqref{Chapt_2_eq_25} can be obtained as a result of substitution of the matrix $K\left(t\right)$ from equation \eqref{Chapt_2_eq_24} by its expression from \eqref{Chapt_2_eq_23}.

\begin{remark}\label{Rem_por_Lk_matr}
Suppose that the conditions of at least one of the Theorems {\normalfont \ref{T_Vikl_Giro_Struct}, \ref{T_Vikl_NP_Struct}} or {\normalfont \ref{T_Vikl_NPG_Struct}} are fulfilled. Then each suitable $L_{k}$-matrix $L\left(t\right)$ can be found as the solution to the matrix differential equation \eqref{Chapt_2_eq_9} supplemented with an initial condition $L(t_{0})=L_{0}$ where $L_{0}$ is an arbitrary matrix form $M_{m}(\mathbb{R}),$ such that $L_{0}L^{T}_{0}=S_{0}.$ Additionally to that, the matrix coefficients of the respective symmetrized system \eqref{Chapt_2_eq_1} can be found via formulas \eqref{Chapt_2_eq_2}.
\end{remark}

\section{Structural transformations of the autonomous systems of second-order ordinary differential equations}\label{A_Section}
Let us consider the two systems of second-order ordinary differential equations
\begin{equation}\label{Chapt_2_2_eq_0}
    \mathbf{\ddot{x}}+A\mathbf{\dot{x}}+B\mathbf{x}=0,\quad A,B\in M_{m}\left(\R\right),
\end{equation}
\begin{equation}\label{Chapt_2_2_eq_-1}
    \mathbf{\ddot{\xi}}+V\mathbf{\dot{\xi}}+W\mathbf{\xi}=0,\quad V,W\in M_{m}\left(\R\right).
\end{equation}
\begin{definition}\label{O_pro_L_ekvival}
We say that the given autonomous system \eqref{Chapt_2_2_eq_0} is
   $L$-equivalent to system \eqref{Chapt_2_2_eq_-1} if there exists a regular on $[0, +\infty)$ matrix $L\left(t\right)$ which satisfies equalities\footnote{Without loss of generality and for the sake of simplicity, in this section we consider the segment $[0, +\infty)$ instead of $[t_{0}, +\infty).$}
\begin{equation}\label{Chapt_2_eq_2'}
   \begin{array}{c}
    V=L^{-1}\left(t\right)\left(2\dot{L}\left(t\right)+AL\left(t\right)\right),\\[1.2em]
    W=L^{-1}\left(t\right)\left(\ddot{L}\left(t\right)+A\dot{L}\left(t\right)+BL\left(t\right)\right),\quad \forall t\in \left[0, +\infty\right).\\
    \end{array}
\end{equation}
\end{definition}

From the first equality of \eqref{Chapt_2_eq_2'} we can easily obtain
   \begin{equation}\label{Chapt_2_2_eq_-2}
    L\left(2t\right)=\exp\left(-A t\right)C \exp\left(V t\right), \quad C\in M_{m}\left(\R\right).
   \end{equation}
It is easy to see that if the matrix $L\left(t\right)$ \eqref{Chapt_2_2_eq_-2} is regular on $[0, +\infty)$ (see definition on page ~\pageref{Page_reg_matr}) then it is an $L_{k}$-matrix for $k=0,1,2.$ Hence, we can see that the notion of the $L_{k}$-equivalence ($k=0,1,2$) for two  autonomous systems according to definition \ref{O_pro_Lk_ekvival} is tantamount to the notion of the $L$-equivalence according to definition \ref{O_pro_L_ekvival}.

 In this section we consider the following symmetrization problems for the autonomous systems of second-order ODEs  \eqref{Chapt_2_2_eq_0}:
\begin{enumerate}\label{Page_zadacha_2}
\item to find an autonomous system \eqref{Chapt_2_2_eq_-1} with $V=V^{T}$ which is $L$-equivalent to the given system \eqref{Chapt_2_2_eq_0} (compare with the EGS problem);
\item to find an autonomous system \eqref{Chapt_2_2_eq_-1} with $W=W^{T}$ which is $L$-equivalent to the given system \eqref{Chapt_2_2_eq_0} (compare with the ENPS problem).
\end{enumerate}

Let us find the necessary and sufficient requirements which have to be imposed on matrices $A, B$ to provide the solvability of the EGS and/or ENPS problems for autonomous system \eqref{Chapt_2_2_eq_0}.

To proceed with this task we have to introduce several convenient notations. We will use the notation $[A,B]$ to describe a commutator of two square matrices $A$ and $B,$ that is,
$$[A,B]=AB-BA.$$ Also, we will use the notation $\{A_{1}A_{2}\ldots A_{n}\}$ to describe a superposition of commutators, that is,
$$\{A_{1}A_{2}\}=[A,B], \quad \{A_{1}A_{2}\ldots A_{n}\}=[\{A_{1}A_{2}\ldots A_{n-1}\}, A_{n}].$$
It is easy to ensure that the commutators obey the following properties:
\begin{equation}\label{Chapt_2_2_eq_1}
    \left[AB, C\right]=\left[A, C\right]B, \quad \forall
A, B, C\in M_{m}(\mathbb{R})\; : \left[B, C\right] =0,
\end{equation}
\begin{equation}\label{Chapt_2_2_eq_2}
    \left[\left[A, B\right], C\right]=\left[A, \left[B,
C\right]\right], \quad \forall A, B, C \in M_{m}(\mathbb{R})\; : \left[A, C\right]=0.
\end{equation}
It is well known that every matrix $A\in M_{m}(\mathbb{R})$ can be expressed in the form of
\begin{equation}\label{Chapt_2_2_eq_3}
    A=T_{A}\;\diag\left[\lambda_{1}\left(A\right)E^{(p_{1})}+H^{(p_{1})},\;\ldots,\;\lambda_{r}\left(A\right)E^{(p_{r})}+H^{(p_{r})}\right]T_{A}^{-1},
\end{equation}
where
\begin{equation}\label{Chapt_2_2_eq_4}
    \lambda_{k}\left(A\right)=\alpha_{k}\left(A\right)+i\;\beta_{k}\left(A\right),\quad
    \alpha_{k}\left(A\right), \beta_{k}\left(A\right) \in \mathbb{R},
\end{equation}
$k=1,2,\ldots,r.$ Here $E^{(p_{k})}$ denotes the identity matrix;  all the elements of square matrix  $H^{(p_{k})}$ are zero except those in the first superdiagonal which are equal to 1. The orders of square matrices  $E^{(p_{k})}$ and $H^{(p_{k})}$ are equal to the power $p_{k}$ of the $k$-th
elementary devisor of matrix $A$ and $T_{A}$ denotes some nonsingular matrix from $M_{m}(\mathbb{R})$ (see, for example, \cite[p. 152]{Gantmaher_v1}).

According to formulas \eqref{Chapt_2_2_eq_3} and \eqref{Chapt_2_2_eq_4} we define \label{Page_AR_AI}
\begin{equation}\label{Chapt_2_2_eq_5}
\left.%
\begin{array}{c}
  A_{R}=T_{A}\;\diag\left[\alpha_{1}\left(A\right)E^{(p_{1})}+H^{(p_{1})},\;\ldots,\alpha_{r}\left(A\right)E^{(p_{r})}+H^{(p_{r})}\right]T_{A}^{-1}, \\[1.2em]
  A_{I}=T_{A}\;\diag\left[i\;\beta_{1}\left(A\right)E^{(p_{1})},\;\ldots,i\;\beta_{r}\left(A\right)E^{(p_{r})}\right]T_{A}^{-1}, \\
\end{array}%
\right.
\end{equation}
then
\begin{equation}\label{Chapt_2_2_eq_6}
    A=A_{R}+A_{I},\; A_{R}A_{I}=A_{I}A_{R}.
\end{equation}
Using the notion of real Jordan canonical form of a real matrix (see \cite[p. 184]{Horn_Johnson}) it is not hard to prove  that if $A\in M_{m}\left(\R\right)$ then $A_{R},\; A_{I}\in M_{m}\left(\R\right).$

Let us consider a Jordan matrix (see, for example, \cite[p. 150]{Horn_Johnson})
\begin{equation}\label{Chapt_2_2_eq_7}
    J_{R}=\diag\left[J_{1}\left(\lambda_{1}\right),\ldots,J_{s}\left(\lambda_{s}\right)\right],
\end{equation}
 where $J_{i}\left(\lambda_{i}\right)$ denotes a Jordan block of size
$m_{i}$ corresponding to the eigenvalue  $\lambda_{i} \in \R,\;i=1,\ldots,s.$ For definiteness we will use the assumption that
\begin{equation}\label{Chapt_2_2_eq_10}
    \lambda_{i}>\lambda_{j}, \; i<j,\;\sum_{i=1}^{s}m_{i}=m.
\end{equation}
 In the above formula $m_{i}$ denotes an algebraic multiplicity of the eigenvalue $\lambda_{i}$ of matrix $J_{R}$ (see \cite[p. 58]{Horn_Johnson}). The following lemma holds true.

\begin{lemma}\label{L_pro_J_formu}
  Suppose that the matrix $L(t)$ is defined by the formula
  \begin{equation}\label{Chapt_2_2_eq_15}
    L(t)=\exp\left(-J_{R}t\right)Q\exp\left(J_{R}t\right), \; t\geq 0,
  \end{equation}
  where $Q\in M_{m}\left(\R\right).$
   Matrix $L(t)$ \eqref{Chapt_2_2_eq_15} is a regular on $[0, +\infty)$ matrix if and only if the matrix $Q$ possesses the following structure:
\begin{equation}\label{Chapt_2_2_eq_8}
    Q=\left[%
\begin{array}{cccc}
  Q_{11} & Q_{12} & \ldots & Q_{1s} \\
  O_{21} & Q_{22} & \ldots & Q_{2s} \\
  \ldots & \ldots & \ldots & \ldots \\
  O_{s1} & O_{s2} & \ldots & Q_{ss} \\
\end{array}%
\right],
\end{equation}  where matrices $Q_{ij}\in M_{m_{i}m_{j}}\left(\R\right)$ satisfy the conditions
\begin{equation}\label{Chapt_2_2_eq_9_1}
    \det\left(Q_{ii}\right)\neq 0,\;
\left[J^{(R)}_{i},Q_{ii}\right]=0
\end{equation}
and $O_{ij}$ denotes a zero-matrix of dimension $m_{i}\times m_{j}, \;\;
i,j=1,2,\ldots,s.$
\end{lemma}

\begin{proof}   Without loss of generality, we consider the case when $s=2,$ that is, when the matrix $J_{R}$ has only two different eigenvalues $\lambda_{1}, \lambda_{2}\in\R,$ $\lambda_{1}>\lambda_{2}$ of an algebraic multiplicity  $m_{1}\geq 0$ and $m_{2}\geq 0$ respectively, $m_{1}+m_{2}=m$.
Let us denote
\begin{equation}\label{Chapt_2_2_eq_11}
    G_{1}=J_{1}\left(0\right),\;G_{2}=J_{2}\left(0\right).
\end{equation}
 From formula \eqref{Chapt_2_2_eq_7}, taking into account notation \eqref{Chapt_2_2_eq_11}, we obtain (see \cite[p. 157]{Gantmaher_v1})
\begin{equation}\label{Chapt_2_2_eq_12}
\begin{array}{c}
  \exp\left( J_{R}t\right)=\diag\left[e^{\lambda_{1}t}\sum\limits_{i=0}^{m_{1}}\frac{1}{i!}t^{i}G_{1}^{i},\;e^{\lambda_{2}t}\sum\limits_{i=0}^{m_{2}}\frac{1}{i!}t^{i}G_{2}^{i}\right], \\ [1.2em]
  \exp\left(- J_{R}t\right)=\left(\exp\left( J_{R}t\right)\right)^{-1}=\diag\left[e^{-\lambda_{1}t}\sum\limits_{i=0}^{m_{1}}\frac{\left(-t\right)^{i}}{i!}G_{1}^{i},\;e^{-\lambda_{2}t}\sum\limits_{i=0}^{m_{2}}\frac{\left(-t\right)^{i}}{i!}G_{2}^{i}\right].
  \end{array}
\end{equation}

{\it Necessity.} Assume that the matrix $L\left(t\right)$ \eqref{Chapt_2_2_eq_15} is a regular on $[0, +\infty)$ matrix. Taking into account formulas \eqref{Chapt_2_2_eq_12} we obtain
\begin{equation}\label{Chapt_2_2_eq_13}
\begin{array}{l}
  L\left(t\right)=\exp\left(-J_{R}t\right)\left[
                                   \begin{array}{cc}
                                     Q_{11} & Q_{12} \\
                                     Q_{21} & Q_{22} \\
                                   \end{array}
                                 \right]
    \exp\left(J_{R}t\right)=\left[
                              \begin{array}{cc}
                                L_{11}\left(t\right) & L_{12}\left(t\right) \\
                                L_{21}\left(t\right) & L_{22}\left(t\right) \\
                              \end{array}
                            \right],
     \\[1.2em]
     L_{11}\left(t\right)=\left(\sum\limits_{i=0}^{m_{1}}\frac{\left(-t\right)^{i}}{i!}G_{1}^{i}\right)Q_{11}\left(\sum\limits_{i=0}^{m_{1}}\frac{t^{i}}{i!}G_{1}^{i}\right),\\[1.2em]
     L_{12}\left(t\right)=e^{\left(\lambda_{2}-\lambda_{1}\right)t}\left(\sum\limits_{i=0}^{m_{1}}\frac{\left(-t\right)^{i}}{i!}G_{1}^{i}\right)Q_{12}\left(\sum\limits_{i=0}^{m_{2}}\frac{t^{i}}{i!}G_{2}^{i}\right),\\[1.2em]
     L_{21}\left(t\right)=e^{\left(\lambda_{1}-\lambda_{2}\right)t}\left(\sum\limits_{i=0}^{m_{2}}\frac{\left(-t\right)^{i}}{i!}G_{2}^{i}\right)Q_{21}\left(\sum\limits_{i=0}^{m_{1}}\frac{t^{i}}{i!}G_{1}^{i}\right),\\[1.2em]
     L_{22}\left(t\right)=\left(\sum\limits_{i=0}^{m_{2}}\frac{\left(-t\right)^{i}}{i!}G_{2}^{i}\right)Q_{22}\left(\sum\limits_{i=0}^{m_{2}}\frac{t^{i}}{i!}G_{2}^{i}\right).\\
\end{array}
\end{equation}
Since the matrices $\sum\limits_{i=0}^{m_{j}}\frac{\left(\pm t\right)^{i}}{i!}G_{j}^{i},\;j=1,2$ are nonsingular, it is easy to see that the matrix $L\left(t\right)$ \eqref{Chapt_2_2_eq_15} has unbounded norm on $[0, \infty)$ unless $Q_{21}=O_{21}$ and matrices $L_{11}\left(t\right),$ $L_{22}\left(t\right),$ whose elements are polynomials of $t,$ are constant. The latter fact implies that
\begin{equation}\label{Chapt_2_2_eq_16}
    L_{11}\left(t\right)=Q_{11},\quad L_{22}\left(t\right)=Q_{22}.
\end{equation}
Particulary, from equalities \eqref{Chapt_2_2_eq_16} it follows that $\det\left(Q_{jj}\right)\neq 0,\; j=1,2.$
Taking into account the equalities $$\left(\sum\limits_{i=0}^{m_{j}}\frac{\left(- t\right)^{i}}{i!}G_{j}^{i}\right)=\left(\sum\limits_{i=0}^{m_{j}}\frac{\left( t\right)^{i}}{i!}G_{j}^{i}\right)^{-1},\;j=1,2,$$ from \eqref{Chapt_2_2_eq_13} and \eqref{Chapt_2_2_eq_16} we obtain
\begin{equation}\label{Chapt_2_2_eq_14}
    Q_{jj}\left(\sum\limits_{i=0}^{m_{j}}\frac{\left( t\right)^{i}}{i!}G_{j}^{i}\right)=\left(\sum\limits_{i=0}^{m_{j}}\frac{\left( t\right)^{i}}{i!}G_{j}^{i}\right)Q_{jj},\; j=1,2,\quad \forall t\geq 0.
\end{equation}
 Equalities \eqref{Chapt_2_2_eq_14} imply that $\left[G_{j}, Q_{jj}\right]=0,$ $j=1,2,$ and we immediately arrive at the conclusion about necessity of commutative equalities in formula \eqref{Chapt_2_2_eq_9_1}. The proof of the necessity is complete.

{\it Sufficiency.} Assume that the matrix $Q$ has a structure described in  \eqref{Chapt_2_2_eq_8}, that is,
 $$Q=\left[
       \begin{array}{cc}
         Q_{11} & Q_{12} \\
         O_{21} & Q_{22} \\
       \end{array}
     \right],
 $$
 and conditions \eqref{Chapt_2_2_eq_9_1} holds true.
Then, taking into account equalities \eqref{Chapt_2_2_eq_12}, we get
\begin{equation}\label{Chapt_2_2_eq_17}
\begin{array}{l}
  L\left(t\right)=\exp\left(-J_{R}t\right)\left[
                                   \begin{array}{cc}
                                     Q_{11} & Q_{12} \\
                                     O_{21} & Q_{22} \\
                                   \end{array}
                                 \right]
    \exp\left(J_{R}t\right)=\left[
                              \begin{array}{cc}
                                Q_{11} & L_{12}\left(t\right) \\
                                O_{21} & Q_{22} \\
                              \end{array}
                            \right],\\[1.2em]
L_{12}\left(t\right)=e^{\left(\lambda_{2}-\lambda_{1}\right)t}\left(\sum\limits_{i=0}^{m_{1}}\frac{\left(-t\right)^{i}}{i!}G_{1}^{i}\right)Q_{12}\left(\sum\limits_{i=0}^{m_{2}}\frac{t^{i}}{i!}G_{2}^{i}\right).\\
\end{array}
\end{equation}
 Conditions \eqref{Chapt_2_2_eq_9_1} together with assumption \eqref{Chapt_2_2_eq_10} imply that the matrix $L\left(t\right)$ \eqref{Chapt_2_2_eq_17} is regular on $[0, +\infty).$ Hence, the sufficiency is proved and the Theorem is proved. \end{proof}

\begin{lemma}\label{L_pro_T_Erugina}
Suppose that $A,\; V,\; C\in M_{m}\left(\R\right).$ If the matrix $L(t),$ defined by formula
\begin{equation}\label{Chapt_2_2_eq_18}
    L(t)=\exp\left(-At\right)C\exp\left(Vt\right),
t\geq 0,
\end{equation}
is regular on $[0, +\infty)$ then the spectra of matrices $A$ and $V$ has the same real part (see definition in \cite[p. 145]{Gantmacher_appl} ), that is, there exists a nonsingular matrix $C_{1}\in M_{m}\left(\R\right)$, such that $$
V_{R}=C_{1}^{-1}A_{R}C_{1}.$$
\end{lemma}

\begin{proof} From the commutativity of matrices $A_{I}$ and $A_{R}$ \eqref{Chapt_2_2_eq_6} it follows that
\begin{equation}\label{Chapt_2_2_eq_19}
    L(t)=\exp\left(-A_{I}t\right)\exp\left(-A_{R}t\right)C\exp\left(V_{R}t\right)\exp\left(V_{I}t\right).
\end{equation}
  Taking into account the definitions of matrices $A_{I}$ and $V_{I}$ and equality \eqref{Chapt_2_2_eq_19} we arrive at the conclusion that the matrix $L\left(t\right)$ is regular on $[0, +\infty)$ if and only if the matrix
$$L_{1}(t)=\exp\left(-A_{R}t\right)C\exp\left(V_{R}t\right)$$ is regular on $[0, +\infty).$
 On the other hand, it is easy to see that the matrix $L_{1}(t)$ represents the general solution to the matrix differential equation (supposing that $C$ represents an arbitrary matrix from space $M_{m}\left(\R\right)$)
\begin{equation}\label{Chapt_2_2_eq_20}
    \frac{d}{d t}L_{1}\left(t\right)=L_{1}\left(t\right)V_{R}-A_{R}L_{1}\left(t\right).
\end{equation}
 In \cite[pp. 121--125]{Gantmaher_v1} it was proved that equation  \eqref{Chapt_2_2_eq_20} possesses a solution  $L_{1}\left(t\right)$ that is a regular on $[0, +\infty)$ matrix if and only if the matrices $A_{R}$ and $V_{R}$ has the same set of elementary devisors. It is known (see \cite[p. 185]{Horn_Johnson}) that if the matrices $A_{R}, V_{R}\in M_{m}\left(\R\right)$ has the same set of elementary devisors then they are similar, furthermore, the similarity matrix $C_{1}$ can be chosen from the space $M_{m}\left(\R\right).$ This completes the proof of the Theorem.
\end{proof}

\begin{lemma}\label{L_pro_suschestv_matritsi}
  Suppose that $A, V, C, Z\in M_{m}\left(\R\right),$
  the matrix
  \begin{equation}\label{Chapt_2_2_eq_21}
     L(t)=\exp\left(-At\right)C\exp\left(Vt\right)
\end{equation} is regular on $[0, +\infty)$ and
\begin{equation}\label{Chapt_2_2_eq_22}
    \left[Z, L(t)C^{-1}\right]=0, \; \forall t\geq 0.
\end{equation}
Then there exists a nonsingular matrix $C_{1}\in M_{m}\left(\R\right),$ such that
\begin{equation}\label{Chapt_2_2_eq_23}
    V_{R}=C_{1}^{-1}A_{R}C_{1},
\end{equation}
\begin{equation}\label{Chapt_2_2_eq_24}
    C^{-1}ZC=C_{1}^{-1}ZC_{1}
\end{equation}
and the matrix
\begin{equation}\label{Chapt_2_2_eq_25}
    L_{1}\left(t\right)=\exp\left(-At\right)C_{1}\exp\left(Vt\right)
\end{equation}
 is a regular on $[0, +\infty)$ matrix satisfying the identity
 \begin{equation}\label{Chapt_2_2_eq_26}
    \left[Z, L_{1}\left(t\right)C_{1}^{-1}\right]=0, \; \forall t\geq
    0.
 \end{equation}
\end{lemma}
\begin{proof} Suppose that the conditions of the Lemma are fulfilled. Then, according to Lemma \ref{L_pro_T_Erugina}, the spectra of matrices $A$ and
$V$ has the same real part. Thus, there exist nonsingular matrices $T_{A}, T_{V}\in M_{m}\left(\R\right),$ such that
\begin{equation}\label{Chapt_2_2_eq_26'}
    \begin{array}{c}
      A=T_{A}\left(J_{R}+I_{A}\right)T_{A}^{-1},\;V=T_{V}\left(J_{R}+I_{V}\right)T_{V}^{-1}, \\[1.2em]
      I_{A}=T_{A}^{-1}A_{I}T_{A},\;I_{V}=T_{V}^{-1}V_{I}T_{V},\;\left[J_{R}, I_{A}\right]=\left[J_{R}, I_{V}\right]=0,
    \end{array}
\end{equation}
where $J_{R}$ is the Jordan matrix defined in \eqref{Chapt_2_2_eq_7}.

Let us consider the matrix $L(t)$ \eqref{Chapt_2_2_eq_21}. Using notation \eqref{Chapt_2_2_eq_26'}, we can rewrite it in the form \begin{equation}\label{Chapt_2_2_eq_27}
\begin{array}{c}
  L(t)=T_{A}\exp\left(-(J_{R}+I_{A})t\right)T_{A}^{-1}C T_{V}\exp\left((J_{R}+I_{V})t\right)T_{V}^{-1}= \\[1.2em]
  =T_{A}\exp\left(-I_{A}t\right)\exp\left(-J_{R}t\right)\Bigl(T_{A}^{-1}CT_{V}\Bigr)\exp\left(J_{R}t\right)\exp\left(I_{V}t\right)T_{V}^{-1}.
\end{array}
\end{equation}

From Lemma \ref{L_pro_J_formu} it follows that
$T_{A}^{-1}CT_{V}=Q,$ where $Q\in M_{m}\left(\R\right)$ is the matrix defined in \eqref{Chapt_2_2_eq_8}.

Formula \eqref{Chapt_2_2_eq_27} leads us to the equality
\begin{equation}\label{Chapt_2_2_eq_28}
    \begin{array}{c}
      T_{A}^{-1}L(t)C^{-1}T_{A}=\exp\left(-I_{A}t\right)\times \\[1.2em]
      \times\diag\left[\exp\left(-J_{1}\left(\lambda_{1}\right)t\right),\ldots,\exp\left(-J_{s}\left(\lambda_{s}\right)t\right)\right]Q\times \\[1.2em]
      \times\diag\left[\exp\left(J_{1}\left(\lambda_{1}\right)t\right),\ldots,\exp\left(J_{s}\left(\lambda_{s}\right)t\right)\right]\exp\left(I_{V}t\right)Q^{-1}. \end{array}
\end{equation}

From equality \eqref{Chapt_2_2_eq_28}, owing to the commutation  properties \eqref{Chapt_2_2_eq_9_1}, we get
\begin{equation}\label{Chapt_2_2_eq_29}
    \begin{array}{c}
      T_{A}^{-1}L(t)C^{-1}T_{A}=\exp\left(-I_{A}t\right)Q_{D}\exp\left(I_{V}t\right)Q^{-1}+\textbf{E}_{1}\left(t\right)= \\[1.2em]
      =\textbf{E}_{0}\left(t\right)+\textbf{E}_{1}\left(t\right),
    \end{array}
\end{equation}
where
$$Q_{D}=\diag\left[Q_{11},\ldots,Q_{ss}\right].$$
It is easy to see that identity \eqref{Chapt_2_2_eq_22} can be rewritten in the form of
\begin{equation}\label{Chapt_2_2_eq_33}
    \left[T_{A}^{-1}ZT_{A}, T_{A}^{-1}L(t)C^{-1}T_{A}\right]=\left[T_{A}^{-1}ZT_{A}, \textbf{E}_{0}\left(t\right)+\textbf{E}_{1}\left(t\right)\right]=0,\; \forall t
\geq 0.
\end{equation}

It is not hard to verify that the elements of matrix
$\textbf{E}_{0}\left(t\right)$ \eqref{Chapt_2_2_eq_29}
can be expressed as linear combinations of functions of type
\begin{equation}\label{Chapt_2_2_eq_30}
  \sin\left(\alpha \;t\right)\pm\cos\left(\alpha \;t\right),\;\alpha, \in \mathbb{R}.
\end{equation}
On the other hand, the elements of matrix $\textbf{E}_{1}\left(t\right)$ (\ref{Chapt_2_2_eq_29}) can be expressed as linear combinations of functions of type
\begin{equation}\label{Chapt_2_2_eq_31}
    t^{p}e^{\rho t}\left(\cos\left(\alpha t\right)\pm\sin\left(\alpha t\right)\right),\; \rho, \alpha \in \R, \rho
    < 0,
\end{equation}
$p\in \N\bigcup\left\{0\right\},$ $p<m.$

If the matrix
$\textbf{E}_{0}\left(t\right)+\textbf{E}_{1}\left(t\right)$
commutates with the constant matrix $T_{A}^{-1}ZT_{A}$ for all $t\geq 0$ (see \eqref{Chapt_2_2_eq_33})
then the same remains true for each of the summands
$\textbf{E}_{0}\left(t\right)$ and $\textbf{E}_{1}\left(t\right)$ separately. Indeed, assume to the contrary that there exists a value $t_{1}\geq 0,$ such that $\left[\textbf{E}_{0}\left(t_{1}\right),
T_{A}^{-1}ZT_{A}\right]\neq 0.$ It is obvious that in this case
 $\left[\textbf{E}_{1}\left(t_{1}\right), T_{A}^{-1}ZT_{A}\right]\neq
0.$ Taking into account the continuity of elements of matrices $\textbf{E}_{0}\left(t\right), \textbf{E}_{1}\left(t\right),$ we obtain
\begin{equation}\label{Chapt_2_2_eq_32}
\begin{array}{c}
  \left[\textbf{E}_{0}\left(t\right),
T_{A}^{-1}ZT_{A}\right]\neq 0,\quad \left[\textbf{E}_{1}\left(t\right), T_{A}^{-1}ZT_{A}\right]\neq
0, \\[1.2em]
  \quad \forall t\in\left[t_{1}-\delta, t_{1}+\delta\right],
\end{array}
\end{equation}
for some sufficiently small positive real number $\delta.$

It is easy to see that each element of the matrix
$\left[\textbf{E}_{0}\left(t\right), T_{A}^{-1}ZT_{A}\right]$
can be expressed as a linear combination of functions of type \eqref{Chapt_2_2_eq_30} and
each element of the matrix $\left[\textbf{E}_{1}\left(t\right),
T_{A}^{-1}ZT_{A}\right]$ can be expressed as a linear combination of functions of type
\eqref{Chapt_2_2_eq_31}. Since the functions of types \eqref{Chapt_2_2_eq_30} and \eqref{Chapt_2_2_eq_31} are linearly independent, we conclude that
$$\begin{array}{c}
    \left[\textbf{E}_{0}\left(t\right)+\textbf{E}_{1}\left(t\right),
T_{A}^{-1}ZT_{A}\right]= \\[1.2em]
    \left[\textbf{E}_{0}\left(t\right),
T_{A}^{-1}ZT_{A}\right]+\left[\textbf{E}_{1}\left(t\right),
T_{A}^{-1}ZT_{A}\right]\neq 0
  \end{array}
$$
for some $t\in\left[t_{1}-\delta, t_{1}+\delta\right].$
Thus, we get a contradictions to condition \eqref{Chapt_2_2_eq_33}. This contradiction proves the incorrectness of our assumption. Therefore we proved the identity
\begin{equation}\label{Chapt_2_2_eq_34}
\left[\textbf{E}_{0}\left(t\right),T_{A}^{-1}ZT_{A}\right]=0,
\forall t\geq 0.
\end{equation}
 Setting $t=0$ in  \eqref{Chapt_2_2_eq_34} we obtain
\begin{equation}\label{Chapt_2_2_eq_35}
    \left[Q_{D}Q^{-1},T_{A}^{-1}ZT_{A}\right]=0.
\end{equation}

Let us construct a matrix $C_{1}\in M_{m}\left(\R\right),$ $\det\left(C_{1}\right)\neq 0$ satisfying equality \eqref{Chapt_2_2_eq_24}.
Using equality \eqref{Chapt_2_2_eq_35} we get
\begin{equation}\label{Chapt_2_2_eq_35'}
\begin{array}{c}
  C^{-1}ZC=C^{-1}T_{A}\left(T_{A}^{-1}ZT_{A}\right)T_{A}^{-1}C=T_{V}\left(T_{V}^{-1}C^{-1}T_{A}\right)\left(T_{A}^{-1}ZT_{A}\right)\times\\[1.2em]
  \times\left(T_{A}^{-1}CT_{V}\right)T_{V}^{-1}
=T_{V}Q^{-1}\left(T_{A}^{-1}ZT_{A}\right)QT_{V}^{-1}= \\[1.2em]
 =T_{V}Q_{D}^{-1}\left(Q_{D}Q^{-1}\right)\left(T_{A}^{-1}ZT_{A}\right)\left(QQ_{D}^{-1}\right)Q_{D}T_{V}^{-1}= \\[1.2em]
  =T_{V}Q_{D}^{-1}\left(T_{A}^{-1}ZT_{A}\right)Q_{D}T_{V}^{-1}=\left(T_{V}Q_{D}^{-1}T_{A}^{-1}\right)Z\left(T_{A}Q_{D}T_{V}^{-1}\right).
\end{array}
\end{equation}
From equalities \eqref{Chapt_2_2_eq_35'} it follows that the matrix $C_{1}$ satisfying condition \eqref{Chapt_2_2_eq_24} can be chosen in the following way
\begin{equation}\label{Chapt_2_2_eq_36}
C_{1}=T_{A}Q_{D}T_{V}^{-1}\in M_{m}\left(\R\right).
\end{equation}

Equality \eqref{Chapt_2_2_eq_23} can be obtained  from the following chain of equalities
$$\begin{array}{c}
    C_{1}^{-1}A_{R}C_{1}=
\left(T_{A}Q_{D}T_{V}^{-1}\right)^{-1}\left(T_{A}J_{R}T_{A}^{-1}\right)T_{A}Q_{D}T_{V}^{-1}= \\[1.2em]
    =T_{V}Q_{D}^{-1}T_{A}^{-1}\left(T_{A}J_{R}T_{A}^{-1}\right)T_{A}Q_{D}T_{V}^{-1}
=T_{V}J_{R}T_{V}^{-1}=V_{R}.
  \end{array}
$$
Let us prove that the matrix $L_{1}\left(t\right)$ \eqref{Chapt_2_2_eq_25} is  regular on $[0, +\infty).$
Taking into account equality \eqref{Chapt_2_2_eq_23} and executing several elementary transformations,  we get
$$\begin{array}{c}
    L_{1}\left(t\right)=\exp\left(-At\right)C_{1}\exp\left(Vt\right)= \\[1.2em]
    =\exp\left(-A_{I}t\right)\exp\left(-A_{R}t\right)C_{1}\exp\left(V_{R}t\right)\exp\left(V_{I}t\right)= \\[1.2em]
    =\exp\left(-A_{I}t\right)C_{1}C_{1}^{-1}\exp\left(-A_{R}t\right)C_{1}\exp\left(V_{R}t\right)\exp\left(V_{I}t\right)= \\[1.2em]
    =\exp\left(-A_{I}t\right)C_{1}\exp\left(-C_{1}^{-1}A_{R}C_{1}t\right)\exp\left(V_{R}t\right)\exp\left(V_{I}t\right)= \\[1.2em]
    =\exp\left(-A_{I}t\right)C_{1}\exp\left(-V_{R}t\right)\exp\left(V_{R}t\right)\exp\left(V_{I}t\right)= \\[1.2em]
    =\exp\left(-A_{I}t\right)C_{1}\exp\left(V_{I}t\right).
  \end{array}
$$
It is easy to see that the matrix $\exp\left(-A_{I}t\right)C_{1}\exp\left(V_{I}t\right)$ is regular on $[0, +\infty).$

Now we intend to prove identity \eqref{Chapt_2_2_eq_26}. Equalities
$$
\begin{array}{c}
  T_{A}^{-1}L_{1}\left(t\right)C_{1}^{-1}T_{A}=
T_{A}^{-1}\Bigl(\exp\left(-A_{I}t\right)\left(T_{A}Q_{D}T_{V}^{-1}\right)\exp\left(V_{I}t\right)\Bigr)\times \\[1.2em]
  \times\Bigl(T_{A}Q_{D}
T_{V}^{-1}\Bigr)^{-1}T_{A}=\exp\left(-I_{A}t\right)Q_{D}\exp\left(I_{V}t\right)Q^{-1}QQ_{D}^{-1}= \\[1.2em]
  =\textbf{E}_{0}\left(t\right)\left(QQ_{D}^{-1}\right).
\end{array}
$$
together with commutation identities \eqref{Chapt_2_2_eq_34} and \eqref{Chapt_2_2_eq_35} immediately lead us to the equalities
$$\begin{array}{c}
  \left[Z,L_{1}\left(t\right)C_{1}^{-1}\right]=\left[T_{A}^{-1}ZT_{A}, T_{A}^{-1}L_{1}\left(t\right)C_{1}^{-1}T_{A}\right]= \\[1.2em]
  =\left[T_{A}^{-1}ZT_{A}, \textbf{E}_{0}\left(t\right)\left(Q\left(Q_{D}\right)^{-1}\right)\right]=0,
\end{array}
$$
which are valid for all $t\geq 0.$

The Theorem is proved.
 \end{proof}

Let us denote by $\mathcal{X}_{n}$ the set containing all the solutions of the system of linear matrix equations
\begin{equation}\label{Nesinchenna_systema}
    \left\{Z A^{(k)}
    X\right\}=0,\;k=0,1,\ldots,n,
\end{equation}
  where $Z,\;A\in M_{m}\left(\R\right)$ are given matrices and $X$ is the unknown square matrix of order $m.$
\begin{theorem} \label{T_pro_konechnost'_beskonechnosti}
There exists a positive integer number $n< m^{2},$ such that
the set equalities
\begin{equation}\label{Chapt_2_2_eq_37}
    \mathcal{X}_{n}=\mathcal{X}_{k},\; k=n+1, n+2,\ldots.
\end{equation}
hold true.
\end{theorem}

\begin{proof} It is not hard to verify that the set
$\mathcal{X}_{n}$ can be represented in the multi-parametric matrix form
\begin{equation}\label{Chapt_2_2_eq_38}
     \mathcal{X}_{n}=\left[\sum_{k=1}^{p_{n}}\chi_{k,i,j}^{(n)}c_{k}\right]_{i,j=1}^{m},
\end{equation}
where $\chi_{k,i,j}^{(n)}$ are constant real coefficients and $c_{k}$
are the arbitrary parameters $k=1,2,\ldots,p_{n}$ (see, for example, \cite[p. 221]{Gantmaher_v1}), $0<p_{n}\leq m^{2}$.

To begin with, we prove that if for some non-negative integer $n\in \N\bigcup\left\{0\right\}$ the set equality
\begin{equation}\label{Chapt_2_2_eq_39}
    \mathcal{X}_{n}=\mathcal{X}_{n+1},
\end{equation}
holds true then
equalities \eqref{Chapt_2_2_eq_37} hold as well. Indeed, equality
\eqref{Chapt_2_2_eq_39} implies that
\begin{equation}\label{Chapt_2_2_eq_40}
\begin{array}{c}
  0=\left\{Z A^{(n+1)}
    \mathcal{X}_{n+1}\right\}=\left[\left[\left\{ZA^{(n)}\right\}, A\right], \mathcal{X}_{n}\right]= \\[1.2em]
  =\left[\left\{ZA^{(n)}\right\},\left[A,
    \mathcal{X}_{n}\right]\right]=\left\{ZA^{(n)}\left[A,
    \mathcal{X}_{n}\right]\right\}.
\end{array}
\end{equation}
From here and below by equality of type \eqref{Chapt_2_2_eq_40} we mean the equality for every element of set
$\mathcal{X}_{n}$. From \eqref{Chapt_2_2_eq_40} it follows that
\begin{equation}\label{Chapt_2_2_eq_41}
    \left[A, \mathcal{X}_{n}\right]\subseteq \mathcal{X}_{n}.
\end{equation}
From equality \eqref{Chapt_2_2_eq_39} and inclusion \eqref{Chapt_2_2_eq_41} it follows that
\begin{equation}\label{Chapt_2_2_eq_42}
\begin{array}{c}
  \left\{ZA^{(n+2)}\X_{n}\right\}=\left[\left[\left\{ZA^{(n+1)}\right\}, A\right], \X_{n}\right]= \\[1.2em]
  =\left[\left\{ZA^{(n+1)}\right\}, \left[A, \X_{n}\right]\right]=\left\{ZA^{(n+1)}\left[A, \X_{n}\right]\right\}=0,
\end{array}
    \end{equation}
i.e., $\X_{n}\subseteq \X_{n+2}.$ On the other hand, from the definition of the set $\X_{n}$ it follows that $\X_{n}\supseteq \X_{n+1} \supseteq \X_{n+2}.$ Therefore, we have
\begin{equation}\label{Chapt_2_2_eq_43}
    \X_{n}=\X_{n+1}=\X_{n+2}.
\end{equation}
 Using reasoning similar to that used above and the method of mathematical induction it is not hard to prove that equality  \eqref{Chapt_2_2_eq_39} implies equalities \eqref{Chapt_2_2_eq_37}.

Now let us prove that the non-negative integer   $n\in\N\bigcup\left\{0\right\},$ mentioned in the Theorem, exists and is less then $m^{2}.$
For this purpose we consider the system of matrix equations
\begin{equation}\label{Chapt_2_2_eq_44}
    \left\{Z A^{(n)} X\right\}=0, \; n=0,1,\ldots,m^{2}-1
\end{equation}
with respect to unknown matrix $X\in M_{m}\left(\R\right)$. If we would show that every solution  $X$ of system \eqref{Chapt_2_2_eq_44} satisfies equalities
\begin{equation}\label{Chapt_2_2_eq_45}
    \left\{Z A^{(n)} X\right\}=0, \; n=m^{2},m^{2}+1,\ldots,
\end{equation}
we will prove the Theorem.

Let us consider the process of solving of system
 \eqref{Chapt_2_2_eq_44}.
Suppose that  $n=0.$ There are only two possible cases (see representation \eqref{Chapt_2_2_eq_38}):

$a)$ $$\left\{Z A \mathcal{X}_{0}\right\}= 0,\; \forall c_{k}\in
\mathbb{R},\; k=1,\ldots, p_{0},$$ that is, we already have found a non-negative integer $n=0,$ such that equalities \eqref{Chapt_2_2_eq_39} hold true. Therefore, as it was proved above,  equalities
\eqref{Chapt_2_2_eq_44}, \eqref{Chapt_2_2_eq_45} hold true for all $X\in \mathcal{X}_{0},$ the process is completed and the Theorem is proved;

$b)$ equality
\begin{equation}\label{Chapt_2_2_eq_46}
    \left\{Z A \mathcal{X}_{0}\right\}=0
\end{equation}
 does not hold true for all possible values of the parameters \linebreak $c_{k}\in \R,$
$k=1,2,\ldots,p_{0}.$ This means that $m^{2}> p_{0}\geq 2,$
because the assumption that $p_{0}=1$ or $p_{0}=m^{2}$ immediately leads us to the equalities $\X_{0}=c_{0}E$ or $Z=aE,\;a\in\R$ respectively and we arrive at the case a). Thus, equality \eqref{Chapt_2_2_eq_46} can be viewed as a system of  $m^{2}$ linear homogeneous equations with respect to the arbitrary parameters $c_{k},\;
k=1,2,\ldots,p_{0}.$ Since this system possesses a non-zero solution, the rank $r_{1}$ of its matrix
 satisfies the two-sided inequality $0<r_{1}<p_{0}.$ If we would solve the given system we will arrive at the matrix (set) $\mathcal{X}_{1}$ \eqref{Chapt_2_2_eq_38}. In addition to that
(see \cite[p. 40--41]{Kurosh}) $p_{1}=p_{0}-r_{1}.$ Therefore,
$p_{1}<p_{0},$ that is, the number of the arbitrary parameters has decreased. Again, there are only two possible cases

$a)$ $$\left\{Z A^{(2)} \mathcal{X}_{1}\right\}=0, \forall
c_{k}\in \mathbb{R},\; k=1,\ldots, p_{1},$$ i.e., the process is completed and the Theorem is proved;

$b)$ equality
\begin{equation}\label{Chapt_2_2_eq_47}
    \left\{Z A^{(2)} \mathcal{X}_{1}\right\}= 0
\end{equation}
  does not hold true for all possible values of the parameters  \linebreak $c_{k}\in \R,\; k=1,2,\ldots,p_{1}.$ It means that  $m^{2}> p_{0}\geq p_{1}+1\geq 3.$
Equality \eqref{Chapt_2_2_eq_47} can be viewed as a system of $m^{2}$ linear homogeneous equations with respect to the parameters $c_{k},\;
k=1,2,\ldots,p_{1}.$ If we would solve this new system we will arrive at the matrix (set)
$\mathcal{X}_{2}$ \eqref{Chapt_2_2_eq_38}. It is obvious that in this case $p_{2}<p_{1},$ that is, the number of the arbitrary parameters has decreased again. And so on.

This process could not contain more than $p_{0}< m^{2}$ steps.
The Theorem is proved.
\end{proof}

\begin{lemma}\label{L_Komut_totozhnist}
    Suppose that $A, V, Z, C \in M_{m}\left(\R\right)$ and $\det\left(C\right)\neq 0.$ Then the commutation identity
\begin{equation}\label{Chapt_2_2_eq_l_1}
      \left[Z, L\left(t\right)C^{-1}\right]=0,\quad \forall t\geq 0, \\[1.2em]
\end{equation}
where
\begin{equation}\label{Chapt_2_2_eq_l_2}
    L\left(2t\right)=\exp\left(-At\right)C\exp\left(Vt\right)
\end{equation}
  holds true if and only if the infinite system of matrix equalities
\begin{equation}\label{Chapt_2_2_eq_l_3}
  \left\{Z A^{(n)}\left(CVC^{-1}-A\right)\right\}=0,\; n=0,1,\ldots.
\end{equation}
holds true.
\end{lemma}

\begin{proof} {\it Necessity.} The matrix-valued function $L\left(t\right)$ \eqref{Chapt_2_2_eq_l_2} satisfies the equalities
\begin{equation}\label{Chapt_2_2_eq_l_4}
\begin{array}{c}
  \left[A, L^{(n)}\left(t\right)C^{-1}\right]=AL^{(n)}\left(t\right)C^{-1}-L^{(n)}\left(t\right)C^{-1}A= \\[1.2em]
  =AL^{(n)}\left(t\right)C^{-1}-L^{(n)}\left(t\right)VC^{-1}+L^{(n)}\left(t\right)VC^{-1}-L^{(n)}\left(t\right)C^{-1}A= \\[1.2em]
  =-2L^{(n+1)}\left(t\right)C^{-1}+L^{(n)}\left(t\right)C^{-1}\left(CVC^{-1}-A\right)=\\[1.2em]
  =-2L^{(n+1)}\left(t\right)C^{-1}+2L^{(n)}\left(t\right)C^{-1}L^{(1)}\left(0\right)C^{-1},\;\;L^{(n)}\left(t\right)\overset{\textmd{def}}{=}\cfrac{d^{n}}{d t^{n}}L\left(t\right)
\end{array}
\end{equation}
$\forall n\in \N\bigcup\left\{0\right\}$ or, that is the same,
\begin{equation}\label{Chapt_2_2_eq_l_5}
    L^{(n+1)}\left(t\right)C^{-1}=L^{(n)}\left(t\right)C^{-1}L^{(1)}\left(0\right)C^{-1}-\frac{1}{2}\left[A, L^{(n)}\left(t\right)C^{-1}\right].
\end{equation}

Suppose that the commutation identity \eqref{Chapt_2_2_eq_l_1} holds. Let us proved that it implies the identities
\begin{equation}\label{Chapt_2_2_eq_l_6}
    \left\{Z A^{(n)} \left(L\left(t\right)C^{-1}\right)\right\}=0,\quad \forall t\geq 0,\;\forall n\in\N\cup \left\{0\right\}.
\end{equation}
  In order to prove this, we will use the method of mathematical induction  with respect to $n.$ If $n=0$ then identity \eqref{Chapt_2_2_eq_l_6} coincides with \eqref{Chapt_2_2_eq_l_1}. If $n=1$ then form identity \eqref{Chapt_2_2_eq_l_1}, using \eqref{Chapt_2_2_eq_l_5} and the properties of commutators \eqref{Chapt_2_2_eq_2}, we get
\begin{equation}\label{Chapt_2_2_eq_l_7}
    \begin{array}{c}
      0=\left[Z, L^{(1)}\left(t\right)C^{-1}\right]=\left[Z, L\left(t\right)C^{-1}L^{(1)}\left(0\right)C^{-1}\right]- \\[1.2em]
      -\cfrac{1}{2}\left[Z, \left[A, L\left(t\right)C^{-1}\right]\right]=-\cfrac{1}{2}\left[\left[Z, A\right] , L\left(t\right)C^{-1}\right]=-\cfrac{1}{2}\left\{Z A  \left(L\left(t\right)C^{-1}\right)\right\}.
    \end{array}
\end{equation}
Equality \eqref{Chapt_2_2_eq_l_7} proves identity \eqref{Chapt_2_2_eq_l_6} with $n=1.$ Let us assume that identity \eqref{Chapt_2_2_eq_l_6} is proved for $n=k\geq 2$ and let us prove it for $n=k+1.$ Using equality \eqref{Chapt_2_2_eq_l_5} and the properties of commutators \eqref{Chapt_2_2_eq_2}, from the latter assumption we obtain
\begin{equation}\label{Chapt_2_2_eq_l_8}
    \begin{array}{c}
      0=\left[\left\{Z A^{(k)}\right\}, L^{(1)}\left(t\right)C^{-1}\right]=\left[\left\{Z A^{(k)}\right\}, L\left(t\right)C^{-1}L^{(1)}\left(0\right)C^{-1}\right]- \\[1.2em]
      -\cfrac{1}{2}\left[\left\{Z A^{(k)}\right\}, \left[A, L\left(t\right)C^{-1}\right]\right]=-\cfrac{1}{2}\left[\left[\left\{Z A^{(k)}\right\}, A\right] , L\left(t\right)C^{-1}\right]=\\[1.2em]
      =-\cfrac{1}{2}\left\{Z A^{(k+1)}  \left(L\left(t\right)C^{-1}\right)\right\}.
    \end{array}
\end{equation}
Therefore, according to the principle of mathematical induction, we have that identity \eqref{Chapt_2_2_eq_l_6} holds for all  $n\in \N \cup \left\{0\right\}.$

Taking into account the arbitrariness of $n\in \N \cup \left\{0\right\}$ in formula \eqref{Chapt_2_2_eq_l_6}, we can obtain equalities \eqref{Chapt_2_2_eq_l_3} via differentiation of identity \eqref{Chapt_2_2_eq_l_6} with respect to $t$ and subsequent substitution $t=0.$

 {\it Sufficiency. }Suppose that equalities \eqref{Chapt_2_2_eq_l_3} hold. Let us prove that they imply  identity \eqref{Chapt_2_2_eq_l_1}. If $n=0$ then from \eqref{Chapt_2_2_eq_l_3} we get
\begin{equation}\label{Chapt_2_2_eq_l_9}
    \left[Z, \left(CVC^{-1}-A\right)\right]=2\left[Z, L^{(1)}\left(0\right)C^{-1}\right]=0.
\end{equation}
If $n=1$ then from \eqref{Chapt_2_2_eq_l_3}, taking into account \eqref{Chapt_2_2_eq_2}, \eqref{Chapt_2_2_eq_l_4} and \eqref{Chapt_2_2_eq_l_9}, we obtain
\begin{equation}\label{Chapt_2_2_eq_l_10}
\begin{array}{c}
  0=\left[\left[Z, A\right], \left(CVC^{-1}-A\right)\right]=2\left[Z, \left[A, L^{(1)}\left(0\right)C^{-1}\right]\right]= \\[1.2em]
  =-4\left[Z, L^{(2)}\left(0\right)C^{-1}\right]+4\left[Z, \left(L^{(1)}\left(0\right)C^{-1}\right)^{2}\right]=-\left(-2\right)^{2}\left[Z, L^{(2)}\left(0\right)C^{-1}\right]. \\
\end{array}
\end{equation}
Let us assume that we already have proved equalities
\begin{equation}\label{Chapt_2_2_eq_l_11}
    \left[Z, L^{(n)}\left(0\right)C^{-1}\right]=0,\; n=1,2,\ldots, k
\end{equation}
for some positive integer $k\geq 2.$

  From equalities \eqref{Chapt_2_2_eq_l_3}, assumption \eqref{Chapt_2_2_eq_l_11}, properties of commutators \eqref{Chapt_2_2_eq_2} and equality \eqref{Chapt_2_2_eq_l_4} we get
 \begin{equation}\label{Chapt_2_2_eq_l_12}
\begin{array}{c}
0=\left\{Z A^{(k)}\left(CVC^{-1}-A\right)\right\}=\left[\left[\left\{ZA^{(k-1)}\right\}, A\right], \left(CVC^{-1}-A\right)\right]= \\[1.2em]
=\left[\left\{Z A^{(k-1)}\right\}, \left[A, CVC^{-1}-A\right]\right]=2\left[\left\{Z A^{(k-1)}\right\}, \left[A, L^{(1)}\left(0\right)C^{-1}\right]\right]=\\[1.2em]
=-4\left[\left\{Z A^{(k-1)}\right\}, L^{(2)}\left(0\right)C^{-1}\right]+4\left[\left\{Z A^{(k-1)}\right\}, \left(L^{(1)}\left(0\right)C^{-1}\right)^{2}\right]=\\[1.2em]
=-4\left[\left[Z A^{(k-1)}\right\}, L^{(2)}\left(0\right)C^{-1}\right]=-4\left[\left[\left\{ZA^{(k-2)}\right\}, A\right], L^{(2)}\left(0\right)C^{-1}\right]=\\[1.2em]
=-4\left[\left\{Z A^{(k-2)}\right\}, \left[A, L^{(0)}\left(0\right)C^{-1}\right]\right]=8\left[\left\{Z A^{(k-2)}\right\}, L^{(3)}\left(0\right)C^{-1}\right]-\\[1.2em]
-8\left[\left\{Z A^{(k-2)}\right\}, L^{(2)}\left(0\right)C^{-1}L^{(1)}\left(0\right)C^{-1}\right]=\ldots \\[1.2em]
\ldots =-\left(-2\right)^{k}\left[\left[Z, A\right], L^{(k)}\left(0\right)C^{-1}\right]=-\left(-2\right)^{k}\left[Z, \left[A, L^{(k)}\left(0\right)C^{-1}\right]\right]=\\[1.2em]
=-\left(-2\right)^{k+1}\left[Z, L^{(k+1)}\left(0\right)C^{-1}\right].
\end{array}
 \end{equation}
Thus, according to the principle of mathematical induction, we have that equalities \eqref{Chapt_2_2_eq_l_11} hold for every non-negative integer $n\in\N\cup \left\{0\right\}.$

 From \eqref{Chapt_2_2_eq_l_2} it follows that
 the matrix series
$$\sum_{n=0}^{\infty}L^{(n)}\left(0\right)C^{-1}\frac{t^{n}}{n!}$$ is dominated by the number series
$$\sum_{n=0}^{\infty}\frac{\left(\left\|A\right\|+\left\|CVC^{-1}\right\|\right)^{n}\left(\frac{t}{2}\right)^{n}}{n!}=\exp\left(\left(\left\|A\right\|+\left\|CVC^{-1}\right\|\right)\frac{t}{2}\right).$$
Thus, the matrix series is uniformly convergent on $\left[0,
+\infty\right)$ and its sum coincides with the matrix
$L\left(t\right)C^{-1}.$ This fact together with equalities \eqref{Chapt_2_2_eq_l_11} immediately lead us to the commutation identity \eqref{Chapt_2_2_eq_l_1}. This completes the proof of the Theorem.
\end{proof}

Now we are in position to prove the main theorem of the paper. It is stated below.
\begin{theorem}[An analogue of the Erugin's theorem] \label{T_main_theorem}
Suppose that $A, B, V, W \in M_{m}\left(\R\right).$ The two systems of second-order differential equations
\begin{equation}\label{Chapt_2_2_eq_48}
    \ddot{x}\left(t\right)+A\;\dot{x}\left(t\right)+B\;x\left(t\right)=0,
\end{equation}
\begin{equation}\label{Chapt_2_2_eq_49}
  \ddot{\xi}\left(t\right)+V\;\dot{\xi}\left(t\right)+W\;\xi\left(t\right)=0
\end{equation}
 are $L$-equivalent
  if and only if there exists a nonsingular matrix  $C\in M_{m}\left(\R\right)$ satisfying conditions
\begin{equation}\label{Chapt_2_2_eq_50}
    V_{R}=C^{-1}A_{R}C,
\end{equation}
\begin{equation}\label{Chapt_2_2_eq_51}
  4W=V^{2}+C^{-1}\left(4B-A^{2}\right)C,
\end{equation}
\begin{equation}\label{Chapt_2_2_eq_52}
  \left\{\left(4B-A^{2}\right)A^{(n)}\left(CVC^{-1}-A\right)\right\}=0,\; n=0,1,\ldots,m^{2}-1.
\end{equation}
\end{theorem}
\begin{proof} \textit{Sufficiency.}
Suppose that for some nonsingular matrix $C\in M_{m}\left(\R\right)$ conditions \eqref{Chapt_2_2_eq_50}, \eqref{Chapt_2_2_eq_51} and \eqref{Chapt_2_2_eq_52} are fulfilled. It is easy to see that the matrix
\begin{equation}\label{Chapt_2_2_eq_54}
   L\left(2t\right)=\exp\left(-At\right)C\exp\left(Vt\right)=\exp\left(-A_{I}t\right)C\exp\left(V_{I}t\right)
\end{equation}
is regular on $[0, +\infty).$
Substituting the matrix $L\left(t\right)$ \eqref{Chapt_2_2_eq_54} into the first equality of \eqref{Chapt_2_eq_2'}
we obtain the identity
\begin{equation}\label{Chapt_2_2_eq_56}
    \left.%
\begin{array}{c}
  L^{-1}\left(t\right)\left(2\dot{L}\left(t\right)+AL\left(t\right)\right)=\\[1.2em]
  =L^{-1}\left(t\right)\left(-AL\left(t\right)+L\left(t\right)V+AL\left(t\right)\right)=V.\\
\end{array}%
\right.
\end{equation}
From the second equality of \eqref{Chapt_2_eq_2'} we get
\begin{equation}\label{Chapt_2_2_eq_57}
    \left.%
\begin{array}{c}
  4L^{-1}\left(t\right)\left(\ddot{L}\left(t\right)+A\dot{L}\left(t\right)+BL\left(t\right)\right)=\\[1.2em]
  =L^{-1}\left(t\right)\left(4B-A^{2}\right)L\left(t\right)+V^{2}=4W. \\
\end{array}%
\right.
\end{equation}
Here we have taken into account that equalities
 \eqref{Chapt_2_2_eq_52}, according to Theorem \ref{T_pro_konechnost'_beskonechnosti} and Lemma \ref{L_Komut_totozhnist}, are equivalent to the commutation identity
\begin{equation}\label{Chapt_2_2_eq_58}
   \left[4B-A^{2}, L\left(t\right)C^{-1}\right]=0,\quad
    \forall t\geq 0.
\end{equation}
Since the regular on $[0, +\infty)$ matrix $L\left(t\right)$ \eqref{Chapt_2_2_eq_54} satisfies conditions \eqref{Chapt_2_eq_2'}, systems \eqref{Chapt_2_2_eq_48} and \eqref{Chapt_2_2_eq_49} are $L$-equivalent. The sufficiency is proved.

\textit{Necessity.}
Suppose that systems \eqref{Chapt_2_2_eq_48} and \eqref{Chapt_2_2_eq_49} are
$L$-equivalent. Then, according to the definition of the $L$-equivalence, there exists a regular on $[0, +\infty)$ matrix $L\left(t\right),$
such that
\begin{equation}\label{Chapt_2_2_eq_60}
    L^{-1}\left(t\right)\left(2\dot{L}\left(t\right)+AL\left(t\right)\right)=V,
\end{equation}
\begin{equation}\label{Chapt_2_2_eq_61}
    L^{-1}\left(t\right)\left(\ddot{L}\left(t\right)+A\dot{L}\left(t\right)+BL\left(t\right)\right)=W,\; \forall t\in\left[0, +\infty\right).
\end{equation}
From \eqref{Chapt_2_2_eq_60} we obtain that
\begin{equation}\label{Chapt_2_2_eq_62}
     L\left(2t\right)=\exp\left(-At\right)C\exp\left(Vt\right),
\end{equation}
where  $C\in M_{m}\left(\R\right),$ $\det\left(C\right)\neq 0.$ Then from
\eqref{Chapt_2_2_eq_61}, using formula \eqref{Chapt_2_2_eq_62} and setting $t=0,$ we obtain equality
\eqref{Chapt_2_2_eq_51} and commutation identity
\eqref{Chapt_2_2_eq_58}.

Since the conditions of Lemma
\ref{L_pro_suschestv_matritsi} are fulfilled,  we can assume that the matrix $C$ is chosen in such a way that identity \eqref{Chapt_2_2_eq_58},
equality \eqref{Chapt_2_2_eq_51} and condition \eqref{Chapt_2_2_eq_50} hold and in addition to that matrix \eqref{Chapt_2_2_eq_62} is regular on $\left[0, +\infty\right).$ From identity \eqref{Chapt_2_2_eq_58}, according to Lemma \eqref{L_Komut_totozhnist}, we get equalities  \eqref{Chapt_2_2_eq_52}.
The necessity is proved and the proof of the Theorem is completed.
\end{proof}

\begin{remark} \label{R_zamechanie_na_sluchay_prostori_spectra}
Suppose that $A, V, C \in M_{m}\left(\R\right)$ and $\det\left(C\right)\neq 0.$ If for some non-negative integer $n\in \N\cup \left\{0\right\}$ the spectrum of the matrix $
 Z_{n}=\left\{\left(4B-A^{2}\right)A^{(n)}\right\}$ is simple, i.e., all the eigenvalues of matrix $Z_{n}$ are different, then conditions \eqref{Chapt_2_2_eq_52} are equivalent to the equalities
\begin{equation}\label{Chapt_2_2_eq_63}
 \left[B, CVC^{-1}-A\right]=0,
\end{equation}
\begin{equation}\label{Chapt_2_2_eq_64}
 \left[A, CVC^{-1}\right]=0.
\end{equation}
\end{remark}

\begin{proof} It is almost obvious that conditions
\eqref{Chapt_2_2_eq_63} and  \eqref{Chapt_2_2_eq_64} imply conditions \eqref{Chapt_2_2_eq_52}.

Suppose that conditions \eqref{Chapt_2_2_eq_52} are fulfilled and for some non-negative integer $n$ the spectrum of matrix
$Z_{n}$ is simple. Then there exists a nonsingular matrix $T,$ such that the matrix ${\textstyle T^{-1}Z_{n}T}$ is diagonal with pairwise different diagonal elements. Thus, (see \cite[p. 221]{Gantmaher_v1}) we have that
\begin{equation}\label{Chapt_2_2_eq_65}
  T^{-1}\left(CVC^{-1}-A\right)T=\verb"diag"\left[\sigma_{1},\ldots,\sigma_{m}\right],\; \sigma_{i}\in\R,\; i\in\overline{1, m}.
\end{equation}
Using Theorem
\ref{T_pro_konechnost'_beskonechnosti} and equalities
\eqref{Chapt_2_2_eq_52} we obtain
$$ 0=\left[\left[Z_{n},A\right],CVC^{-1}-A\right]=\left[Z_{n},\left[A,CVC^{-1}-A\right]\right].$$
Applying the same reasoning as above to the latter equalities we arrive at the following representation, which is similar to \eqref{Chapt_2_2_eq_65}:
\begin{equation}\label{Chapt_2_2_eq_66}
\begin{array}{c}
  T^{-1}\left[A,CVC^{-1}-A\right]T=T^{-1}ATT^{-1}\left(CVC^{-1}-A\right)T- \\[1.2em]
  -T^{-1}\left(CVC^{-1}-A\right)TT^{-1}AT
  =TAT^{-1}\diag\left[\sigma_{1},\ldots, \sigma_{m}\right]- \\[1.2em]
  - \diag\left[\sigma_{1},\ldots, \sigma_{m}\right]TAT^{-1}=\diag\left[\tau_{1},\ldots, \tau_{m}\right],\;\; \tau_{i}\in\R,\; i\in\overline{1, m}. \\[1.2em]
\end{array}
\end{equation}
 It is not hard to verify that representation \eqref{Chapt_2_2_eq_66} implies equality $\left[A, CVC^{-1}-A\right]=0$ which, on its own account, implies equality
\eqref{Chapt_2_2_eq_64}. Additionally to that equality \eqref{Chapt_2_2_eq_63} obviously follows from \eqref{Chapt_2_2_eq_52}. The proof is completed.
\end{proof}

Though Theorem \ref{T_main_theorem} gives us the necessary and sufficient conditions providing that systems \eqref{Chapt_2_2_eq_48} and \eqref{Chapt_2_2_eq_49} are equivalent ($L$-equivalent, to be precise), conditions \eqref{Chapt_2_2_eq_50}, \eqref{Chapt_2_2_eq_51} and \eqref{Chapt_2_2_eq_52} of the Theorem do not possess the property of symmetry, which is one of the main properties of an equivalence relation. However, this is only the matter of the wording. In that form the theorem about  $L$-equivalence will be useful in the further sections of the paper. Theorem \ref{T_main_theorem} can be reformulated in the ``symmetric'' form presented below.

\begin{theorem}[An analogue of the Erugin's theorem in the ``symmetric'' form]
Suppose that $A, B, V, W\in M_{m}\left(\R\right).$ The two systems of second-order differential equations \eqref{Chapt_2_2_eq_48} and
\eqref{Chapt_2_2_eq_49} are $L$-equivalent if and only if there exists a nonsingular matrix $C\in M_{m}\left(\R\right),$ such that
\begin{equation}\label{Chapt_2_2_eq_67}
    C V_{R}=A_{R}C,
\end{equation}
\begin{equation}\label{Chapt_2_2_eq_68}
 C\left(4W-V^{2}\right)=\left(4B-A^{2}\right)C,
\end{equation}
\begin{equation}\label{Chapt_2_2_eq_69}
\left.%
\begin{array}{c}
    \left\{\left(4B-A^{2}\right)A^{(n)}\right\}\left(CV-AC\right) =\left(CV-AC\right)\left\{\left(4W-V^{2}\right)V^{(n)}\right\}, \\[1.2em]
 n=0,1,2,\ldots,m^{2}-1.
\end{array}%
\right.
\end{equation}
\end{theorem}
\begin{proof} To prove the Theorem it is enough to show that conditions \eqref{Chapt_2_2_eq_50} -- \eqref{Chapt_2_2_eq_52} are equivalent to conditions \eqref{Chapt_2_2_eq_67} -- \eqref{Chapt_2_2_eq_69}. It is easy to see that condition \eqref{Chapt_2_2_eq_50} is equivalent to condition \eqref{Chapt_2_2_eq_67}, as well as condition \eqref{Chapt_2_2_eq_51} is equivalent to condition \eqref{Chapt_2_2_eq_68}.

Taking into account \eqref{Chapt_2_2_eq_68}, from equalities \eqref{Chapt_2_2_eq_69} with $n=0$ we obtain the equalities
 \begin{equation}\label{Chapt_2_2_eq_70}
 \begin{array}{c}
   \left(4B-A^{2}\right)\left(CVC^{-1}-A\right)=\left(CVC^{-1}-A\right)C\left(4W-V^{2}\right)C^{-1}= \\[1.2em]
    =\left(CVC^{-1}-A\right)\left(4B-A^{2}\right)
 \end{array}
 \end{equation}
 which lead us to condition \eqref{Chapt_2_2_eq_52} with $n=0.$
Multiplying equality \eqref{Chapt_2_2_eq_70} on $C^{-1}$ from the left and on $C$ from the right and rearranging the summands,  we get
\begin{equation}\label{Chapt_2_2_eq_71}
    [C^{-1}\left(4B-A^{2}\right)C, V]=[\left(4W-V^{2}\right), V]=C^{-1}\left[\left(4B-A^{2}\right), A\right]C.
\end{equation}
 From equalities \eqref{Chapt_2_2_eq_69} with $n=1,$ taking into account \eqref{Chapt_2_2_eq_68} and \eqref{Chapt_2_2_eq_71}, we obtain the equalities
\begin{equation}\label{Chapt_2_2_eq_72}
    \begin{array}{c}
      \left\{\left(4B-A^{2}\right)A\right\}\left(CVC^{-1}-A\right)=\left(CVC^{-1}-A\right)C\left\{\left(4W-V^{2}\right)V\right\}C^{-1}= \\[1.2em]
      =\left(CVC^{-1}-A\right)\left\{\left(4B-A^{2}\right)V\right\}
    \end{array}
\end{equation}
 which lead us to condition \eqref{Chapt_2_2_eq_52} with $n=1.$
Multiplying equality \eqref{Chapt_2_2_eq_72} on $C^{-1}$ from the left and on $C$ from the right and rearranging the summands, we get
$$\begin{array}{c}
    \left(C^{-1}\left\{\left(4B-A^{2}\right)A\right\}C\right)V-V\left(C^{-1}\left\{\left(4B-A^{2}\right)A\right\}C\right)= \\[1.2em]
    =C^{-1}\left\{\left(4B-A^{2}\right)A^{(2)}\right\}C.
  \end{array}
$$
Combining the latter equality with \eqref{Chapt_2_2_eq_71} we obtain
\begin{equation}\label{Chapt_2_2_eq_73}
    \left\{\left(4W-V^{2}\right)V^{(2)}\right\}=C^{-1}\left\{\left(4B-A^{2}\right)A^{(2)}\right\}C.
\end{equation}

Therefore we have proved that the first two equalities of \eqref{Chapt_2_2_eq_52} (with $n=0,1$) are equivalent to the first two equalities of \eqref{Chapt_2_2_eq_69} (with $n=0,1$) respectively. Besides that we have proved the auxiliary equalities \eqref{Chapt_2_2_eq_71} and \eqref{Chapt_2_2_eq_73}. Let us assume that for some positive integer $k,$ $2< k<m^{2}-1$ we have proved that the first $k$ equalities of \eqref{Chapt_2_2_eq_52} are equivalent to the first $k$ equalities of \eqref{Chapt_2_2_eq_69} (with $n=0,1,\ldots,k-1$) respectively  and the auxiliary equality (similar to \eqref{Chapt_2_2_eq_73})
\begin{equation}\label{Chapt_2_2_eq_74}
    \left\{\left(4W-V^{2}\right)V^{(k)}\right\}=C^{-1}\left\{\left(4B-A^{2}\right)A^{(k)}\right\}C
\end{equation}
holds.
Then from equalities \eqref{Chapt_2_2_eq_69} with $n=k,$ taking into account \eqref{Chapt_2_2_eq_74}, we obtain equalities
\begin{equation}\label{Chapt_2_2_eq_75}
    \begin{array}{c}
      \left\{\left(4B-A^{2}\right)A^{(k)}\right\}\left(CVC^{-1}-A\right)=\left(CVC^{-1}-A\right)C\left\{\left(4W-V^{2}\right)V^{(k)}\right\}C^{-1}= \\[1.2em]
      =\left(CVC^{-1}-A\right)\left\{\left(4B-A^{2}\right)V^{(k)}\right\}
    \end{array}
\end{equation}
which lead us to condition \eqref{Chapt_2_2_eq_52} with $n=k.$
In addition to that, multiplying equality \eqref{Chapt_2_2_eq_75} on $C^{-1}$ from the left and on $C$ from the right and rearranging the summands, we obtain
$$\begin{array}{c}
    \left(C^{-1}\left\{\left(4B-A^{2}\right)A^{(k)}\right\}C\right)V-V\left(C^{-1}\left\{\left(4B-A^{2}\right)A^{(k)}\right\}C\right)= \\[1.2em]
    =C^{-1}\left\{\left(4B-A^{2}\right)A^{(k+1)}\right\}C.
  \end{array}
$$
 Combining the latter equality with assumption \eqref{Chapt_2_2_eq_74} we get
\begin{equation}\label{Chapt_2_2_eq_76}
    \left\{\left(4W-V^{2}\right)V^{(k+1)}\right\}=C^{-1}\left\{\left(4B-A^{2}\right)A^{(k+1)}\right\}C.
\end{equation}
Therefore we have proved that the first $k+1$ equalities of \eqref{Chapt_2_2_eq_52} are equivalent to the first $k+1$ equalities of \eqref{Chapt_2_2_eq_69} (with $n=0,1,\ldots,k$) respectively.  Also, we have proved the auxiliary equality \eqref{Chapt_2_2_eq_76}. According to the principle of mathematical induction we can conclude that equalities \eqref{Chapt_2_2_eq_52} are equivalent to equalities \eqref{Chapt_2_2_eq_69}, provided that condition \eqref{Chapt_2_2_eq_68} holds. This completes the proof of the Theorem. \end{proof}

\section{Consequences from Theorem \ref{T_main_theorem}}\label{Conseq_Section}

Below we have stated several consequences from Theorem \ref{T_main_theorem} that are related to the question of symmetrization of the matrix differential equation (or, in other words, the system of differential equations)
\begin{equation}\label{Chapt_2_2_eq_77}
    J\mathbf{\ddot{x}}+\left(D+G\right)\mathbf{\dot{x}}+\left(P+\Pi\right)\mathbf{x}=0,
\end{equation}
where $J, D, G, P, \Pi\in M_{m}\left(\R\right),$ $J=J^{T}>0,$ $D=D^{T},$ $\Pi=\Pi^{T},$ $G=-G^{T},$ $P=-P^{T}.$
Let us denote
\begin{equation}\label{A_and_B_notation}
A=J^{-\frac{1}{2}}\left(D+G\right)J^{-\frac{1}{2}},\quad B=J^{-\frac{1}{2}}\left(P+\Pi\right)J^{-\frac{1}{2}}.
\end{equation}

\begin{corollary}\label{N_1}

    Suppose that there exist a symmetric matrix $V\in M_{m}\left(\R\right)$ and a nonsingular matrix $C\in M_{m}\left(\R\right)$ satisfying conditions
\begin{equation}\label{Chapt_2_2_eq_79}
    \left[A,\;CVC^{-1}\right]=ACVC^{-1}-CVC^{-1}A=0,
\end{equation}
\begin{equation}\label{Chapt_2_2_eq_80}
    \left[B,\;A-CVC^{-1}\right]=B\left(A-CVC^{-1}\right)-\left(A-CVC^{-1}\right)B=0,
\end{equation}
\begin{equation}\label{Chapt_2_2_eq_81}
    CV_{R}=A_{R}C.
\end{equation}
Then the autonomous equation \eqref{Chapt_2_2_eq_77} is $L$-equivalent to the autonomous equation
\begin{equation}\label{Chapt_2_2_eq_78}
    \mathbf{\ddot{\xi}}+V\mathbf{\dot{\xi}}+W\mathbf{\xi}=0,\quad V, W\in M_{m}\left(\R\right),
\end{equation}
\begin{equation}\label{Chapt_2_2_eq_82}
    W=\frac{1}{4}V^{2}+C^{-1}\left(B-\frac{1}{4}A^{2}\right)C,
\end{equation}
containing no gyroscopic structures $(V=V^{T})$.
\end{corollary}

\begin{corollary}\label{N_2}
 Suppose that there exist matrices $V, C \in M_{m}\left(\R\right),$ $\det\left(C\right)\neq 0$ satisfying conditions \eqref{Chapt_2_2_eq_79}--\eqref{Chapt_2_2_eq_81} and
\begin{equation}\label{Chapt_2_2_eq_83}
    V^{2}-\left(V^{2}\right)^{T}+C^{-1}ZC-C^{T}Z^{T}\left(C^{-1}\right)^{T}=0
\end{equation}
where
$$Z=\left(B-\frac{1}{4}A^{2}\right).$$
Then the autonomous equation \eqref{Chapt_2_2_eq_77} is $L$-equivalent to the autonomous equation \eqref{Chapt_2_2_eq_78}, \eqref{Chapt_2_2_eq_82}, containing no non-conservative positional structures  $(W=W^{T}).$
\end{corollary}

\begin{corollary}\label{N_3}
 Suppose that there exist a symmetric matrix  $V\in M_{m}\left(\R\right)$ and a nonsingular matrix $C\in M_{m}\left(\R\right)$ satisfying conditions \eqref{Chapt_2_2_eq_79}--\eqref{Chapt_2_2_eq_81} together with the equality
\begin{equation}\label{Chapt_2_2_eq_84}
    C^{-1}ZC-C^{T}Z^{T}\left(C^{-1}\right)^{T}=0.
\end{equation}
Then the autonomous equation \eqref{Chapt_2_2_eq_77} is $L$-equivalent to the ``symmetric'' autonomous equation \eqref{Chapt_2_2_eq_78}, \eqref{Chapt_2_2_eq_82} $(W=W^{T},\; V=V^{T}).$
\end{corollary}
\begin{corollary}\label{N_4}
If for some non-negative integer  $n$ the spectrum of the matrix
$$Z_{n}=\left\{\left(4B-A^{2}\right)A^{(n)}\right\}$$
is simple then the conditions of Corollaries {\normalfont \ref{N_1}--\ref{N_3}} are the necessary ones (not only sufficient!).
\end{corollary}

 Combining Theorem \ref{T_main_theorem} with the theorems of Kelvin -- Tait -- Chetayev it is not hard to prove the following theorem  that can be viewed as a generalization of the Mingori's  \cite{Mingori_Eng} and M\"{u}ller's \cite{Muller} theorems.
\begin{theorem}\label{T_Uzag_Mingori}
Suppose that the matrices $V, C\in M_{m}\left(\R\right),$ $\det\left(C\right)\neq 0,$ \mbox{$V+V^{T}>0$} satisfy conditions
\begin{equation}\label{Chapt_2_2_eq_85}
    C V_{R}=A_{R}C,
\end{equation}
\begin{equation}\label{Chapt_2_2_eq_86}
  \left\{\left(4B-A^{2}\right)A^{(n)}\left(CVC^{-1}-A\right)\right\}=0,\; n=0,1,\ldots,m^{2}-1,
\end{equation}
\begin{equation}\label{Chapt_2_2_eq_87}
    V^{2}-\left(V^{2}\right)^{T}+C^{-1}\left(4B-A^{2}\right)C-C^{T}\left(4B-A^{2}\right)^{T}\left(C^{-1}\right)^{T}=0.
\end{equation}
If the symmetric matrix
\begin{equation}\label{Chapt_2_2_eq_88}
  W=\frac{1}{4}V^{2}+C^{-1}\left(B-\frac{1}{4}A^{2}\right)C
\end{equation}
is positive definite then the null solution of system  \eqref{Chapt_2_2_eq_77} is asymptotically stable  {\normalfont (}in the sense of Lyapunov {\normalfont )} and if matrix \eqref{Chapt_2_2_eq_88} is nonsingular and has at least one negative eigenvalue then the null solution of system \eqref{Chapt_2_2_eq_77} is unstable  {\normalfont (}in the sense of Lyapunov{\normalfont )}.
\end{theorem}
It is easy to see that if $P=0$ then conditions \eqref{Chapt_2_2_eq_85} -- \eqref{Chapt_2_2_eq_87} can be satisfied once we take $V=A,$ $C=E.$ In this case we would have that $W=\Pi.$ This means that Theorem \ref{T_Uzag_Mingori} can be considered as a generalization  of the 3-rd and 4-th theorems of Kelvin -- Tait -- Chetayev (see \cite{Merkin}).

 It is not hard to verify that the conditions of the Mingori's \cite{Mingori_Eng} and M\"{u}ller's \cite{Muller} theorems implies conditions \eqref{Chapt_2_2_eq_85} -- \eqref{Chapt_2_2_eq_87}. However the following example shows that the converse of above proposition is not correct in general.

 {\bf Example.} Assume that
 \begin{equation}\label{Chapt_2_2_eq_89}
 \begin{array}{c}
   A=\diag\left[A_{1}, A_{2}\right], \quad B=\left[
                                               \begin{array}{cc}
                                                 B_{1} & b_{5}E^{(2)} \\
                                                 b_{5}E^{(2)} & B_{2} \\
                                               \end{array}
                                             \right], \\ [1.2em]
   A_{1}=T\left[
                                                    \begin{array}{cc}
                                                      a_{1} & a_{2} \\
                                                      -a_{2} & a_{1} \\
                                                    \end{array}
                                                  \right]T^{-1},\;A_{2}=T\left[
                                                    \begin{array}{cc}
                                                      a_{3} & a_{4} \\
                                                      -a_{4} & a_{3} \\
                                                    \end{array}
                                                  \right]T^{-1},\\[1.2em]
                                                   B_{1}=T\left[
                                                    \begin{array}{cc}
                                                      b_{1} & b_{2} \\
                                                      -b_{2} & b_{1} \\
                                                    \end{array}\right]T^{-1},\;
                                                   B_{2}=T\left[
                                                    \begin{array}{cc}
                                                      b_{3} & \frac{a_{3}b_{2}}{a_{1}} \\
                                                      -\frac{a_{3}b_{2}}{a_{1}} & b_{3} \\
                                                    \end{array}\right]T^{-1},\\[1.2em]
                                                    J=\diag\left[1,1,1,1\right], \quad T=\left[
                                                        \begin{array}{cc}
                                                          1 & 1 \\
                                                          0 & 1 \\
                                                        \end{array}
                                                      \right],\quad E^{(2)}=\left[
                                                                              \begin{array}{cc}
                                                                                1 & 0 \\
                                                                                0 & 1 \\
                                                                              \end{array}
                                                                            \right].
 \end{array}
 \end{equation}
 Then in terms of matrix coefficients of equation \eqref{Chapt_2_2_eq_77} we have
\begin{equation}\label{Chapt_2_2_eq_90}
\begin{array}{c}
  D=\diag\left[D_{1}, D_{2}\right], \quad G=\diag\left[G_{1}, G_{2}\right], \\[1.2em]
  D_{1}=\left[
          \begin{array}{cc}
            a_{1}-a_{2} & \frac{a_{2}}{2} \\
            \frac{a_{2}}{2} & a_{2}+a_{1} \\
          \end{array}
        \right],\; D_{2}=\left[
                           \begin{array}{cc}
                             a_{3}-a_{4} & \frac{a_{4}}{2} \\
                             \frac{a_{4}}{2} & a_{4}+a_{3} \\
                           \end{array}
                         \right],\\[1.2em]
   G_{1}=-\frac{3a_{2}}{2}S^{(2)},\quad G_{2}=-\frac{3a_{4}}{2}S^{(2)},\quad S^{(2)}=\left[
                                                                                     \begin{array}{cc}
                                                                                       0 & -1 \\
                                                                                       1 & 0 \\
                                                                                     \end{array}
                                                                                   \right],\\[1.2em]
   \Pi=\left[
         \begin{array}{cc}
           \Pi_{1} & b_{5}E^{(2)} \\
           b_{5}E^{(2)} & \Pi_{2} \\
         \end{array}
       \right],\quad P=-\diag\left[\frac{3b_{2}}{2}S^{(2)}, \frac{3a_{3}b_{2}}{2a_{1}}S^{(2)}\right],\\[1.2em]
    \Pi_{1}=\left[
              \begin{array}{cc}
                b_{1}-b_{2} & \frac{b_{2}}{2} \\
                \frac{b_{2}}{2} & b_{2}+b_{1} \\
              \end{array}
            \right],\; \Pi_{2}=\left[
                                 \begin{array}{cc}
                                   \frac{b_{3}a_{1}-b_{2}a_{3}}{a_{1}} & \frac{a_{3}b_{2}}{2a_{1}} \\
                                   \frac{a_{3}b_{2}}{2a_{1}} & \frac{a_{3}b_{2}+b_{3}a_{1}}{a_{1}} \\
                                 \end{array}
                               \right].
\end{array}
\end{equation}
Both, the Mingori's \cite{Mingori_Eng} and M\"{u}ller's \cite{Muller} theorems demand the commutativity of the matrices $P$ and $D$. However, it is easy to verify that for the matrices $P$ and $D$ \eqref{Chapt_2_2_eq_90} this condition is not fulfilled in general. Thus, we can't use the results of the mentioned theorems for the stability investigation of system \eqref{Chapt_2_2_eq_77}, \eqref{Chapt_2_2_eq_90}. On the other hand, the matrices
\begin{equation}\label{Chapt_2_2_eq_91}
\begin{array}{c}
  V=\diag\left[V_{1}, V_{2}\right],\quad C=\diag\left[T, T\right], \\[1.2em]
  V_{1}=a_{1}E^{(2)}-\left(a_{2}-\frac{2b_{2}}{a_{1}}\right)S^{(2)},\quad V_{2}=a_{3}E^{(2)}-\left(a_{4}-\frac{2b_{2}}{a_{1}}\right)S^{(2)},
\end{array}
\end{equation}
satisfy conditions \eqref{Chapt_2_2_eq_79} -- \eqref{Chapt_2_2_eq_81} of Theorem  \ref{T_Uzag_Mingori}, according to which the matrix $W$ \eqref{Chapt_2_2_eq_88} can be expressed in the form of
\begin{equation}\label{Chapt_2_2_eq_92}
  W=\left[
        \begin{array}{cc}
          w_{1}E^{(2)} & b_{5}E^{(2)} \\
          b_{5}E^{(2)} & w_{2}E^{(2)} \\
        \end{array}
      \right],\; w_{1}=\frac{a_{1}a_{2}b_{2}-b_{2}^{2}+a_{1}^{2}b_{1}}{a_{1}^{2}},\;w_{2}=\frac{a_{1}a_{4}b_{2}-b_{2}^{2}+a_{1}^{2}b_{3}}{a_{1}^{2}}.
\end{equation}
 The conditions of Sylvester’s criterion (see \cite[c. 99]{Bellman}), when applied to the matrix
 $W$ \eqref{Chapt_2_2_eq_92}, lead us to the inequalities
 \begin{equation}\label{Chapt_2_2_eq_93}
    w_{1}>0,\quad w_{2}>0.
 \end{equation}
 Inequalities \eqref{Chapt_2_2_eq_93} together with the conditions $a_{1}>0, a_{3}>0$ (providing that the matrix $V$ \eqref{Chapt_2_2_eq_91} is positive definite) describe the region of the asymptotical stability of the null solution of system \eqref{Chapt_2_2_eq_77}, \eqref{Chapt_2_2_eq_90}.

\section{On the interconnection between the notions of the  $L_{k}$-equivalence and the equivalence in the sense of Lyapunov}\label{Connection_Section}
It is well known that systems \eqref{Chapt_2_eq_0} and \eqref{Chapt_2_eq_1} can be rewritten in the form of
\begin{equation}\label{Chapt_2_2_eq_103}
    \frac{d}{dt}\left[%
\begin{array}{c}
  \mathbf{x} \\
  \mathbf{\dot{x}} \\
\end{array}%
\right]=\mathbf{A}^{\ast}\left(t\right)\left[%
\begin{array}{c}
  \mathbf{x} \\
  \mathbf{\dot{x}} \\
\end{array}%
\right], \quad \mathbf{A}^{\ast}\left(t\right)=\left[%
\begin{array}{cc}
  O & E \\
  -B\left(t\right) & -A\left(t\right) \\
\end{array}%
\right],
\end{equation}
and
\begin{equation}\label{Chapt_2_2_eq_97}
    \frac{d}{dt}\left[%
\begin{array}{c}
  \mathbf{\xi} \\
  \mathbf{\dot{\xi}} \\
\end{array}%
\right]=\mathbf{V}^{\ast}\left(t\right)\left[%
\begin{array}{c}
  \mathbf{\xi} \\
  \mathbf{\dot{\xi}} \\
\end{array}%
\right],\quad \mathbf{V}^{\ast}\left(t\right)=\left[%
\begin{array}{cc}
  O & E \\
  -W\left(t\right) & -V\left(t\right) \\
\end{array}%
\right],
\end{equation}
respectively. Suppose that systems \eqref{Chapt_2_2_eq_103} and \eqref{Chapt_2_2_eq_97} are connected by the transformation
\begin{equation}\label{Chapt_2_2_eq_110}
    \left[
      \begin{array}{c}
        \mathbf{x} \\
        \mathbf{\dot{x}} \\
      \end{array}
    \right]=\L\left(t\right)\left[
      \begin{array}{c}
        \mathbf{\xi} \\
        \mathbf{\dot{\xi}} \\
      \end{array}
    \right],\quad \L\left(t\right)=\left[
                       \begin{array}{cc}
                         L_{11}\left(t\right) & L_{12}\left(t\right) \\
                         L_{21}\left(t\right) & L_{22}\left(t\right) \\
                       \end{array}
                     \right], t\in\left[t_{0}, +\infty\right),
\end{equation}
$L_{ij}\left(t\right)\in M_{m}\left(C^{1}\left[t_{0}, +\infty\right]\right).$ It is not hard to verify that this would be the case if and only if the equalities
\begin{equation}\label{Chapt_2_2_eq_101}
  \begin{array}{c}
    \dot{L}_{11}\left(t\right)-L_{12}\left(t\right)W\left(t\right)-L_{21}\left(t\right)= 0, \\[1.2em]
    \dot{L}_{12}\left(t\right)-L_{12}\left(t\right)V\left(t\right)+L_{11}\left(t\right)-L_{22}\left(t\right)= 0,\;  \\
  \end{array}
\end{equation}
\begin{equation}\label{Chapt_2_2_eq_102}
    \begin{array}{c}
      B\left(t\right)L_{11}\left(t\right)+A\left(t\right)L_{21}\left(t\right)=L_{22}\left(t\right)W\left(t\right)-\dot{L}_{21}\left(t\right), \\[1.2em]
      B\left(t\right)L_{12}\left(t\right)+A\left(t\right)L_{22}\left(t\right)=-L_{21}\left(t\right)+L_{22}\left(t\right)V\left(t\right)-\dot{L}_{22}\left(t\right) \end{array}
\end{equation}
hold true $\forall t\in\left[t_{0}, +\infty\right).$

In  accordance with the definition of the equivalence in the sense of Lyapunov of two systems of first-order ODEs that was given in \cite[p. 118]{Gantmaher_v2} we can introduce the same notion for the case of second-order systems.
\begin{definition}\label{O_pro_ekv_v_vensi_Lyapunova}
  We say that the systems of second-order ODEs \eqref{Chapt_2_eq_0} and \eqref{Chapt_2_eq_1} are equivalent in the sense of Lyapunov if there exists a Lyapunov matrix {\normalfont (see definition in \cite[p. 117]{Gantmaher_v2})} $\L\left(t\right)$ \eqref{Chapt_2_2_eq_110} satisfying conditions \eqref{Chapt_2_2_eq_101}, \eqref{Chapt_2_2_eq_102}.
\end{definition}

Let us assume that the matrix $\L\left(t\right)$ \eqref{Chapt_2_2_eq_110} satisfies conditions
\begin{equation*}
  L_{12}\left(t\right)=O,\quad L_{11}\left(t\right)=L\left(t\right)\in M_{m}\left(C^{2}\left[t_{0}, +\infty\right)\right),
\end{equation*}
\begin{equation*}
    \inf\limits_{t\in\left[t_{0}, +\infty\right]}\left|\det\left(L\left(t\right)\right)\right|> 0,\;
    \sup\limits_{t\in \left[t_{0}, +\infty\right)}\left\|\frac{d^{k}}{d t^{k}}L\left(t\right)\right\|<+\infty,\; \forall k\in\overline{0,2}.
\end{equation*}
Then from equalities \eqref{Chapt_2_2_eq_101} we immediately obtain that $L_{21}\left(t\right)=\dot{L}_{11}\left(t\right),$ $L_{22}\left(t\right)=L\left(t\right);$ transformation \eqref{Chapt_2_2_eq_110} reduces to the form
\begin{equation}\label{Chapt_2_2_eq_105}
\left[%
\begin{array}{c}
  \mathbf{x} \\
  \mathbf{\dot{x}} \\
\end{array}%
\right]=\left[%
\begin{array}{cc}
  L\left(t\right) & O \\
  \dot{L}\left(t\right) & L\left(t\right) \\
\end{array}%
\right]\left[%
\begin{array}{c}
  \mathbf{\xi} \\
  \mathbf{\dot{\xi}} \\
\end{array}%
\right]
\end{equation}
and represents a Lyapunov transformation (see definition in \cite[p. 117]{Gantmaher_v2}); conditions \eqref{Chapt_2_2_eq_102} reduce to conditions \eqref{Chapt_2_eq_2}, i.e.,
\begin{equation}\label{Chapt_2_2_eq_108}
\begin{array}{c}
  V\left(t\right)=L^{-1}\left(t\right)\left(2\dot{L}\left(t\right)+A\left(t\right)L\left(t\right)\right), \\[1.2em]
  W\left(t\right)=L^{-1}\left(t\right)\left(\ddot{L}\left(t\right)+A\left(t\right)\dot{L}\left(t\right)+B\left(t\right)L\left(t\right)\right).
\end{array}
\end{equation}
 Thus, we can conclude that if systems \eqref{Chapt_2_eq_0} and \eqref{Chapt_2_eq_1} are $L_{2}$-equivalent according to Definition \ref{O_pro_Lk_ekvival} then they are equivalent in the sense of Lyipunov according to Definition \ref{O_pro_ekv_v_vensi_Lyapunova}. However, it is almost obvious that the converse of above proposition is not correct in general. It is easy to see that the notion of the equivalence in the sense of Lyapunov includes the notions of the $L_{2}$-equivalence (see Definition \ref{O_pro_Lk_ekvival}) and the $L$-equivalence (see Definition \ref{O_pro_L_ekvival}) as partial cases.
Therefore, when we consider the possibility of using structural transformations to aid the investigation of stability of the null solution of system \eqref{Chapt_2_eq_0}, we inevitably arrive at the following {\it general problems of symmetrization}:
\begin{enumerate}
\item {\it for the given system \eqref{Chapt_2_eq_0}, find a Lyapunov matrix $\L\left(t\right)$ \eqref{Chapt_2_2_eq_110} and matrices  $V\left(t\right), W\left(t\right)\in M_{m}\left(C\left[t_{0}, +\infty\right)\right)$ which satisfy the symmetry conditions  $V\left(t\right)=V^{T}\left(t\right)$ and/or $W\left(t\right)=W^{T}\left(t\right)$ together with equalities \eqref{Chapt_2_2_eq_101}, \eqref{Chapt_2_2_eq_102} $\forall t\in\left[t_{0}, +\infty\right);$}
\item {\it for the given autonomous system \eqref{Chapt_2_eq_0}, i.e, $A\left(t\right)=A\in M_{m}\left(\R\right),$ $B\left(t\right)=B\in M_{m}\left(\R\right),$ find a Lyapunov matrix $\L\left(t\right)$ \eqref{Chapt_2_2_eq_110} and matrices $V\left(t\right)=V\in M_{m}\left(\R\right),$ $W\left(t\right)=W\in M_{m}\left(\R\right)$ which satisfy the symmetry conditions $V=V^{T}$ and/or $W=W^{T}$ together with equalities \eqref{Chapt_2_2_eq_101}, \eqref{Chapt_2_2_eq_102} $\forall t\in\left[0, +\infty\right).$}
\end{enumerate}

In the case when systems \eqref{Chapt_2_eq_0} and \eqref{Chapt_2_eq_1} are autonomous, i.e., $\mathbf{A}^{\ast}\left(t\right)=\mathbf{A}^{\ast}\in M_{2m}\left(\R\right),$ $\mathbf{V}^{\ast}\left(t\right)=\mathbf{V}^{\ast}\in M_{2m}\left(\R\right),$ the necessary and sufficient conditions providing that they are equivalent in the sense of Lyapunov where found by Erugin (see the Erugin's theorem in \cite[p. 145]{Gantmacher_appl}): {\it Two systems \eqref{Chapt_2_2_eq_103} and \eqref{Chapt_2_2_eq_97} ($\mathbf{A}^{\ast}$ and $\mathbf{V}^{\ast}$ are constant matrices of the same order) are equivalent in the sense of Lyapunov if and only if the matrices $\mathbf{A}^{\ast}$ and  $\mathbf{V}^{\ast}$ have one and the same real part of the spectrum or , in other words, there exists a nonsingular matrix $\mathbf{C}\in M_{2m}\left(\R\right),$ satisfying equality $$\mathbf{A}^{\ast}_{R}=\mathbf{C}\mathbf{V}^{\ast}_{R}\mathbf{C}^{-1}.$$}

Evidently, in general case, to check whether the conditions of the Erugin's theorem are fulfilled could be as difficult as to solve both systems \eqref{Chapt_2_2_eq_103} and \eqref{Chapt_2_2_eq_97} directly.
However, Theorems \ref{T_Vikl_Giro_Struct} -- \ref{T_Vikl_NPG_Struct}, \ref{T_main_theorem} indicate that in some cases the question about equivalence in the sese of Lyapunov of two systems \eqref{Chapt_2_2_eq_103} and \eqref{Chapt_2_2_eq_97} can be answered without necessity to solve them.

Let as suppose that $V\left(t\right), W\left(t\right)\in M_{m}\left(C^{1}\left[t_{0}, +\infty\right)\right).$ Then substituting the expressions for matrices $L_{21}\left(t\right)$ and  $L_{22}\left(t\right)$ obtained from equations \eqref{Chapt_2_2_eq_101} into equation \eqref{Chapt_2_2_eq_102}, we get the following system of second-order matrix differential equations with respect to the unknown matrices $L_{11}\left(t\right),\; L_{12}\left(t\right):$
\begin{equation}\label{Chapt_2_2_eq_106'}
\begin{array}{c}
  \cfrac{d}{d t}\left(\dot{Z}\left(t\right)+Z\left(t\right)\mathbf{V}^{\ast}\left(t\right)\right)
  + \left(\dot{Z}\left(t\right)+Z\left(t\right)\mathbf{V}^{\ast}\left(t\right)\right)\mathbf{V}^{\ast}\left(t\right)+B\left(t\right)Z\left(t\right)+\\[1.2em]
                                                           +A\left(t\right)\left(\dot{Z}\left(t\right)+Z\left(t\right)\mathbf{V}^{\ast}\left(t\right)\right)=0,\quad Z\left(t\right)=\left[L_{11}\left(t\right), L_{12}\left(t\right)\right].
\end{array}
\end{equation}
Thus, we arrive at the conclusion that systems  \eqref{Chapt_2_eq_0} and \eqref{Chapt_2_eq_1} are equivalent in the sense of Lyapunov if and only if system \eqref{Chapt_2_2_eq_106'} possesses a solution $Z\left(t\right)$  satisfying conditions
\begin{equation}\label{Chapt_2_2_eq_106}
    \begin{array}{c}
      \sup\limits_{t\in\left[t_{0}, +\infty\right)}\bigg\|\cfrac{d^{k}}{d t^{k}}\L\left(t\right)\bigg\|<+\infty,\; k=0,1,\quad \inf\limits_{t\in\left[t_{0}, +\infty\right)}\left|\det\left(\L\left(t\right)\right)\right|>0,\\[1.2em]
      L_{21}\left(t\right)=\dot{L}_{11}\left(t\right)-L_{12}\left(t\right)W\left(t\right), \\ [1.2em]
      L_{22}\left(t\right)=\dot{L}_{12}\left(t\right)-L_{12}\left(t\right)V\left(t\right)+L_{11}\left(t\right). \\
    \end{array}
\end{equation}

The {\it general problems of symmetrization} (GPS) stated above have not been studied in this paper. However, on my opinion, the problem of finding necessary and sufficient conditions for solvability of the GPS can be interesting from both practical and theoretical points of view. This problem is significantly more complicated then the problem of finding necessary and sufficient conditions for solvability of the EGS and/or ENPS problems (see definitions on pp. \pageref{Page_zadacha_1} and \pageref{Page_zadacha_2}). The main reason for that is the significant complexity of conditions \eqref{Chapt_2_2_eq_106'}, \eqref{Chapt_2_2_eq_106} for finding the matrices $\L\left(t\right), \;V\left(t\right), \; W\left(t\right).$ On the other hand, as it was mentioned above, in some cases to solve the GPS for the given system the one should be able to determine the Jordan canonical form of the system's matrix (see the conditions of the Erugin's theorem). Evidently, in this case the using of structural transformations can't facilitate the stability  investigation of the null solution of the system.

\section{Application of the structural transformations to the stability investigation of dynamical systems}\label{Priklad_Section}
{\bf The stability of rotary motion of a rigid body suspended on a string.}
Let us consider the symmetrization problem for the system of second-order differential equations describing the perturbed motion of a heavy, symmetric rigid body suspended to the stationary point $O$ by the inextensible weightless string. We assume that the string attaches to the body at the point $S$ lying on the body's symmetry axis. We denote the distance between point $S$ and the center of mass of the body by $a,$ and the length of the string  by $b.$
It is known that the rotary motion of the body can be approximately described   by the following equations (see equations (2.8) and (2.9) from
\cite{Karepetyan_Lagutina}):
\begin{equation}\label{Chapt_2_5_eq_1}
    \left\{%
\begin{array}{l}
  J_{1}\ddot{x}_{1}+\lambda\dot{x}_{1}+cx_{1}-(2J_{1}-J_{3})\omega\dot{x}_{2}+
  (\lambda_{1}-\lambda)\omega x_{2}+mgax_{3}=0, \\
  J_{1}\ddot{x}_{2}+\lambda\dot{x}_{2}+cx_{2}+(2J_{1}-J_{3})\omega\dot{x}_{1}-
  (\lambda_{1}-\lambda)\omega x_{1}+mgax_{4}=0, \\
  mb^{2}\ddot{x}_{3}+mb(g-b\omega^{2})x_{3}-2mb^{2}\omega \dot{x}_{4}+mgax_{1}=0, \\
    mb^{2}\ddot{x}_{4}+mb(g-b\omega^{2})x_{4}+2mb^{2}\omega \dot{x}_{3}+mgax_{2}=0, \\
\end{array}%
\right.
\end{equation}
where $c=mga\left(\varepsilon+1\right)+\left(J_{3}-J_{1}\right)\omega^{2},
\; a=b\varepsilon,$ $\lambda=f D_{1},$ $\lambda_{1}=f D_{3}.$ In equations \eqref{Chapt_2_5_eq_1} by $\omega>0$  we denote the angular velocity of rotation of the body, by $m$ --- the mass of the body, by $g$ --- the free fall acceleration, and by $J^{\ast}=\diag\left[J_{1}, J_{1}, J_{3}\right]$ --- the central tensor of inertia of the body. The authors of \cite{Karepetyan_Lagutina} assume that the body is effected by the dissipative moment  $M_{d}=-f D \omega,$ where $D=\diag\left[D_{1}, D_{1}, D_{3}\right],$ $D_{1}>0, D_{3}>0, f>0.$ Additionally, we assume that $2J_{1}-J_{3}\neq 0,$ $J_{1}>0.$

It is easy to see that system \eqref{Chapt_2_5_eq_1} is equivalent to system \eqref{Chapt_2_2_eq_77} with
$$J=\diag\left[J_{1}\; E^{(2)},mb^{2}\; E^{(2)}\right],\ \ \
D=\diag\left[\lambda\; E^{(2)}, O^{(2)}\right],$$
$$G=\diag\left[(2J_{1}-J_{3})\omega\; S^{(2)}, 2mb^{2}\omega\; S^{(2)}\right], \ \ \
$$
$$\Pi=\left[%
\begin{array}{ll}
  c\;E^{(2)} & mga\; E^{(2)} \\
  mga\;E^{(2)} & mb(g-b\omega^{2})\;E^{(2)} \\
\end{array}%
\right],$$ $$P=\diag\left[-(\lambda_{1}-\lambda)\omega\;S^{(2)},
 O^{(2)}\right],\ \
\ $$
$$E^{(2)}=\left[%
\begin{array}{cc}
  1 & 0 \\
  0 & 1 \\
\end{array}%
\right],\ \ \ S^{(2)}=\left[%
\begin{array}{cc}
  0 & -1 \\
  1 & 0 \\
\end{array}%
\right], $$ where $O^{(2)}$ denotes the square zero matrix of order $2$.
Furthermore, using notation \eqref{A_and_B_notation} we get
\begin{equation}\label{Chapt_2_5_eq_2}
\left.%
\begin{array}{c}
  A=J^{-\frac{1}{2}}\left(D+G\right)J^{-\frac{1}{2}}=D_{1}+G_{1}= \\[1.2em]
 =\diag\left[\frac{1}{J_{1}}\left[%
\begin{array}{cc}
  \lambda & -(2J_{1}-J_{3})\omega \\
  (2J_{1}-J_{3})\omega & \lambda \\
\end{array}%
\right],2\omega\left[%
\begin{array}{cc}
  0 & -1 \\
  1 & 0 \\
\end{array}%
\right]\right], \\
\end{array}%
\right.
\end{equation}

\begin{equation}\label{Chapt_2_5_eq_2_1}
\begin{array}{c}
  B=J^{-\frac{1}{2}}\left(P+\Pi\right)J^{-\frac{1}{2}}=P_{1}+\Pi_{1}= \\[1.2em]
   =\left[%
\begin{array}{cc}
  cJ_{1}^{-1}E^{(2)}-J_{1}^{-1}(\lambda_{1}-\lambda)\omega\;S^{(2)} & \frac{g a\sqrt{m}}{b\sqrt{J_{1}}}E^{(2)} \\
  \frac{g a\sqrt{m}}{b\sqrt{J_{1}}}E^{(2)} & \left(gb^{-1}-\omega^{2}\right)E^{(2)} \\
\end{array}%
\right].
\end{array}
\end{equation}

Let us find out the sufficient conditions in terms of the parameters of system \eqref{Chapt_2_5_eq_1} which provide that the system is equivalent to some other system does not containing the gyroscopic structures and (or) non-conservative positional structures.

\textbf{The elimination of the gyroscopic structures.}
 It is easy to verify that the spectrum of matrix $Z_{0}=\left(4B-A^{2}\right)$ \eqref{Chapt_2_5_eq_2}, \eqref{Chapt_2_5_eq_2_1} is simple. Therefore, according to Corollary \ref{N_4} the conditions of Corollary \ref{N_1} are necessary and sufficient simultaneously. Let us check whether the conditions of Corollary~\ref{N_1} are fulfilled.  The matrix $CV_{1}C^{-1}$ satisfying condition
\eqref{Chapt_2_2_eq_79} can be expressed in the form of
\begin{equation}\label{Chapt_2_5_eq_3}
 CVC^{-1}=\diag\left[\left[%
\begin{array}{cc}
  v_{11} & v_{12} \\
  -v_{12} & v_{11} \\
\end{array}%
\right], \left[%
\begin{array}{cc}
  v_{33} & v_{34} \\
  -v_{34} & v_{33} \\
\end{array}%
\right]\right].
\end{equation}
Taking into account representation \eqref{Chapt_2_5_eq_3} and the fact that matrix V is real we arrive at the conclusion that condition
\eqref{Chapt_2_2_eq_81} can be satisfied if and only if
\begin{equation}\label{Chapt_2_5_eq_4}
    v_{11}=\frac{\lambda}{J_{1}},\; v_{33}=0.
\end{equation}
 Taking into account \eqref{Chapt_2_5_eq_3} and \eqref{Chapt_2_5_eq_4}, from condition \eqref{Chapt_2_2_eq_80}
we can find that
\begin{equation}\label{Chapt_2_5_eq_5}
    CVC^{-1}=\diag\left[\left[%
\begin{array}{cc}
  \lambda J_{1}^{-1} & v_{12}+\omega J_{3}J_{1}^{-1} \\
  -v_{12}-\omega J_{3}J_{1}^{-1} & \lambda J_{1}^{-1} \\
\end{array}%
\right], \left[%
\begin{array}{cc}
  0 & v_{12} \\
  -v_{12} & 0 \\
\end{array}%
\right]\right].
\end{equation}
 From formula \eqref{Chapt_2_5_eq_5} it follows that the matrix  $V$ is a symmetric matrix if and only if
\begin{equation}\label{Chapt_2_5_eq_6}
J_{3}=0, \;v_{12}=0.
\end{equation}
Thus, the gyroscopic structures can be excluded from system \eqref{Chapt_2_5_eq_1} if and only if $J_{3} = 0.$

  We can assume that the condition $J_{3}=0$ is satisfied if the value of the inertia moment $J_{3}$ is fairly small in comparison with the value of $2J_{1}.$ This can be the case when the body is heavy and has a shape of a cylinder with a very small transverse section.

Following to the Sommerfeld- Greenhill concept we can set
 $\lambda=\mu J_{1},\quad \lambda_{1}=\mu J_{3},$ where $\mu$ is a small constant coefficient depending on the environment characteristics. Returning to the case of a heavy cylinder with a very small transverse section, we can assume that $\lambda_{1}=0$.

 \textbf{The elimination of the non-conservative positional structures.} As it was shown above, conditions \eqref{Chapt_2_2_eq_79} -- \eqref{Chapt_2_2_eq_81}
led us to representation \eqref{Chapt_2_5_eq_5}. Let us take $C=E^{(4)}.$
To satisfy condition  \eqref{Chapt_2_2_eq_83} we take $v_{12}=-\frac{2\omega \lambda_{1}}{\lambda}$ and according to Corollary \ref{N_2}, whose conditions are fulfilled, obtain the matrix coefficients of equation \eqref{Chapt_2_2_eq_78}
\begin{equation}\label{Chapt_2_5_eq_7}
    \begin{array}{c}
      V=\diag\left[V^{(1)}, V^{(2)}\right], \\[1.2em]
      V^{(1)}=\frac{1}{\lambda J_{1}}\left[%
\begin{array}{cc}
  \lambda^{2} & J_{3}\omega \lambda-2\omega\lambda_{1}J_{1} \\
  -J_{3}\omega \lambda+2\omega\lambda_{1}J_{1} & \lambda^{2}  \\
\end{array}%
\right], \\[1.2em]
      V^{(2)}=\frac{2\omega \lambda_{1}}{\lambda}\left[%
\begin{array}{cc}
  0 & -1 \\
  1 & 0 \\
\end{array}%
\right], \\[1.2em]
      W=\left[%
\begin{array}{cc}
  \left(\omega^{2}\frac{\lambda_{1}}{\lambda}\left(\frac{J_{3}}{J_{1}}-\frac{\lambda_{1}}{\lambda}\right)+\frac{mga}{J_{1}}(\varepsilon+1)\right)E^{(2)} & \frac{\sqrt{m}g\varepsilon}{\sqrt{J_{1}}}E^{(2)} \\
  \frac{\sqrt{m}g\varepsilon}{\sqrt{J_{1}}}E^{(2)} & \left(\frac{g}{b}-\frac{\lambda_{1}^{2}\omega^{2}}{\lambda^{2}}\right)E^{(2)} \\
\end{array}%
\right].
    \end{array}
\end{equation}
We see that, according to Corollary \ref{N_2}, the elimination of the non-conservative positional structures is possible without any additional restrictions on the parameters of system \eqref{Chapt_2_5_eq_1}.

Once the non-conservative positional structures are eliminated, we can try to find out the region of the asymptotic stability of the null solution of system  \eqref{Chapt_2_5_eq_1}. Since systems \eqref{Chapt_2_5_eq_1} and  \eqref{Chapt_2_2_eq_78}, \eqref{Chapt_2_5_eq_7} are $L$-equivalent, that is, equivalent in the sense of Lyapunov, their regions of the asymptotic stability coincide. Let us find the region of the asymptotic stability of system \eqref{Chapt_2_2_eq_78}, \eqref{Chapt_2_5_eq_7}.

Unfortunately, the matrix $V$ \eqref{Chapt_2_5_eq_7} is not a positive definite matrix, that is, the conditions of Theorem \ref{T_Uzag_Mingori} are not fulfilled. However, this problem can be overcame.
First of all let us emphasize the fact that if the parameters of system \eqref{Chapt_2_5_eq_1} are chosen in such a way that $\det\left(W\right)=0$ then the null solution of system \eqref{Chapt_2_5_eq_1} is unstable. Thus, we can assume that the matrix $W$ is nonsingular.

It is easy to verify that if
\begin{equation}\label{Chapt_2_5_eq_8}
    W>0
\end{equation}
then function $V\left(\mathbf{\xi}\right)=\mathbf{\dot{\xi}}^{T}\mathbf{\dot{\xi}}+\mathbf{\xi}^{T}W\mathbf{\xi}$ where $\mathbf{\xi}=\mathbf{\xi}(t)$ represents an arbitrary solution of system \eqref{Chapt_2_2_eq_78}, \eqref{Chapt_2_5_eq_7},
satisfies the conditions of the Krasovsky theorem on asymptotic stability (see, for example, \cite[p. 42]{Merkin}). On the other hand if the symmetric matrix $W$ \eqref{Chapt_2_5_eq_7} is nonsingular and has at least one negative eigenvalue then the function $-V\left(\mathbf{\xi}\right)$
satisfies the conditions of the Krasovsky theorem on instability (see, for example, \cite[p. 51]{Merkin}). Thus, we can conclude that condition \eqref{Chapt_2_5_eq_8} describes the required region of asymptotic stability.

The conditions of Sylvester’s criterion, when applied to the matrix $W$ \eqref{Chapt_2_5_eq_7},
lead us to the following system of inequalities:
 \begin{equation}\label{Chapt_2_5_eq_9}
P > 0, \quad PS-R^{2} > 0,
\end{equation}
 where
$$P=\left(\omega^{2}\frac{\lambda_{1}}{\lambda}\left(\frac{J_{3}}{J_{1}}-\frac{\lambda_{1}}{\lambda}\right)+\frac{mga}{J_{1}}(\varepsilon+1)\right),$$
$$S=\left(\frac{g}{b}-\frac{\lambda_{1}^{2}\omega^{2}}{\lambda^{2}}\right),\;R=\frac{\sqrt{m}g\varepsilon}{\sqrt{J_{1}}}.$$

Returning to the case of a heavy cylinder with a very small transverse section and setting
$\lambda_{1}=0$ we see that
 $$   P=\frac{m g a}{J_{1}}\left(1+\frac{a}{b}\right),\quad
    S=\frac{g}{b},\quad R=\frac{g a \sqrt{m}}{b \sqrt{J_{1}}}. $$
Therefore, the first inequality of \eqref{Chapt_2_5_eq_9} is fulfilled and the second one reduces to the form
$$\frac{m g^{2}}{J_{1}}\frac{a}{b}>0.$$ It is worth to emphasize that conditions \eqref{Chapt_2_5_eq_9} are in good agreement with the similar conditions obtained in \cite{Kosh_Storozh_MTT}. On the other hand, a sophisticated method proposed in paper \cite{Karepetyan_Lagutina}  for the stability investigation of the null solution of system \eqref{Chapt_2_5_eq_1} results in a set of inequalities which do not describe the region of asymptotic stability of the system (contrary to the expectations of the authors of the paper). The reason for that is an essential error introduced in  \cite{Karepetyan_Lagutina} by the authors.

\section{Conclusions}\label{Concl_Section}
In the present paper we have extended and generalized the results of a series of papers devoted to the question of stability investigation of the null solution of systems of second-order ODEs via the stability preserving structural transformations. The series was started with D.L. Mingori \cite{Mingori_Eng} and then continued by Von P. C.~M\"{u}ller \cite{Muller},  V.N.~Koshlyakov \cite{Koshlyakov_PMM_2000}, V.N.~Koshlyakov and V.L.~Makarov  \cite{Koshlyakov_S_MTT_1998,Kosh_Makarov_UMJ_2000,Kosh_Makarov_PMM_2001,Kosh_Makarov_PMM_2007}, V.N.~Koshlyakov  and V.A.~Storozhenko  \cite{Kosh_Storozh_MTT}.

In the paper we have found the necessary and sufficient conditions providing that a given autonomous (non-autonomous) system of second-order ODEs is equivalent in the sense of Lyapunov to some autonomous (non-autonomous) system of second-order ODEs which do not contain gyroscopic and/or non-conservative positional structures. Theorem \ref{T_Uzag_Mingori} proved in the paper generalizes the 3-rd and 4-th Kelvin -- Tait -- Chetayev theorems as well as the Mingori's \cite{Mingori_Eng} and M\"{u}ller's \cite{Muller} theorems. In Section \ref{Priklad_Section} it was shown that the theoretical results presented in the paper can be successfully applied to the stability investigation of the null solution of systems of second-order ODEs.

\bibliographystyle{plain}

\bibliography{references_stat}
\end{document}